\documentclass[10pt]{article}

\usepackage{amsmath}
\usepackage{amsfonts}
\usepackage{amsthm}
\usepackage{amssymb}
\usepackage[english]{babel}
\usepackage{graphicx}
\usepackage[all]{xy}

\newtheorem{theorem}{Theorem}[section]
\newtheorem{lemma}{Lemma}[section]

\newtheorem{remark}{Remark}[section]

\newcommand{\mR}{\mathbb{R}}
\newcommand{\mC}{\mathbb{C}}
\newcommand{\mN}{\mathbb{N}}
\newcommand{\mE}{\mathbb{E}}
\newcommand{\mZ}{\mathbb{Z}}
\newcommand{\mS}{\mathbb{S}}

\newcommand{\mQ}{\mathbb{Q}}

\newcommand{\cH}{\mathcal{H}}

\newcommand{\cP}{\mathcal{P}}
\newcommand{\cC}{\mathcal{C}}

\newcommand{\cO}{\mathcal{O}}

\newcommand{\cU}{\mathcal{U}}

\newcommand{\ux}{\underline{x}}

\newcommand{\uxb}{\underline{x} \grave{}}

\newcommand{\mg}{\mathfrak{g}}

\newcommand{\osp}{\mathfrak{osp}(m|2n)}
\newcommand{\sosp}{\mathfrak{so}(m)\oplus \mathfrak{sp}(2n)}

\begin{document}
\title{The orthosymplectic superalgebra in harmonic analysis}

\author{K.\ Coulembier\thanks{Ph.D. Fellow of the Research Foundation - Flanders (FWO), E-mail: {\tt Coulembier@cage.ugent.be}}}

\date{\small{Department of Mathematical Analysis}\\
\small{Faculty of Engineering -- Ghent University\\ Krijgslaan 281, 9000 Gent,
Belgium}\\
\vspace{4mm}
\small{School of Mathematics and Statistics}\\
\small{University of Sydney\\ Sydney, Australia}}

\maketitle

\begin{abstract}
We introduce the orthosymplectic superalgebra $\osp$ as the algebra of Killing vector fields on Riemannian superspace $\mR^{m|2n}$ which stabilize the origin. The Laplace operator and norm squared on $\mR^{m|2n}$, which generate $\mathfrak{sl}_2$, are orthosymplectically invariant, therefore we obtain the Howe dual pair $(\mathfrak{osp}(m|2n),\mathfrak{sl}_2)$. We study the $\osp$-representation structure of the kernel of the Laplace operator. This also yields the decomposition of the supersymmetric tensor powers of the fundamental $\osp$-representation under the action of $\mathfrak{sl}_2\times\osp$. As a side result we obtain information about the irreducible $\osp$-representations $L_{(k,0,\cdots,0)}^{m|2n}$. In particular we find branching rules with respect to $\mathfrak{osp}(m-1|2n)$. We also prove that integration over the supersphere is uniquely defined by its orthosymplectic invariance.    
 
\end{abstract}

\textbf{MSC 2000 :}   17B10, 58C50\\
\noindent
\textbf{Keywords :} Howe dual pair, orthosymplectic superalgebra, not completely reducible representations, supersymmetric tensor product, Cartan product

\section{Introduction}

In recent work, we have been developing a new approach to the study of supergeometry, by means of harmonic analysis, see e.g. \cite{MR2539324, CDBS3, Mehler, DBE1} and references therein. We consider flat superspace $\mR^{m|2n}$ generated by $m$ commuting or bosonic variables and $2n$ anti-commuting or fermionic variables. The main feature of this approach is the introduction of a super Laplace operator $\nabla^2$ and a super norm squared $R^2$. These generate the Lie algebra $\mathfrak{sl}_2$ and are invariant under the action of the product of the orthogonal and the symplectic group $O(m)\times Sp(2n)$. This leads to the pair $\left(\sosp,\mathfrak{sl}_2\right)$, which has been studied in \cite{DBE1}, as a generalization of the Howe dual pair $(\mathfrak{so}(m),\mathfrak{sl}_2)$ for harmonic analysis on $\mR^m$, see \cite{MR0986027}. However, this pair does not satisfy the requirements to be an actual Howe dual pair. This manifests itself in several ways that are listed at the beginning of Section \ref{introOSp}. In short, there are too many invariant functions and linear maps on $\mR^{m|2n}$ and the spaces of spherical harmonics (i.e. polynomial null-solutions of $\nabla^2$ of a fixed degree) are not irreducible $\sosp$-representations. This implies that the polynomials do not have a multiplicity free decomposition into irreducible pieces under the joint action of the dual pair, which is a fundamental property of Howe dual pairs, see \cite{MR2028498, MR0986027, MR1272070, MR1827078}. A similar situation occurs in the theory of Dunkl operators, with pair $(G,\mathfrak{sl}_2)$ where $G <O(m)$ is a Coxeter group, see e.g. \cite{MR1827871}. These problems will be solved by considering the orthosymplectic superalgebra $\mathfrak{osp}(m|2n)$. The formalism of harmonic analysis will prove to be very useful, although many of the problems can be posed in a purely representation-theoretical way. The space of polynomials on $\mR^{m|2n}$, denoted by $\cP$, corresponds to the supersymmetric tensor powers of the fundamental $\osp$-representation $V=L^{m|2n}_{(1,0,\cdots,0)}$. Hence $\cP\cong T(V)=\oplus_{k=0}^\infty \odot^k V$, with $V$ the $m|2n$-dimensional super vectorspace. One of the main results of this paper is therefore a complete decomposition of this $\osp$-representation with the help of the $\mathfrak{sl}_2$-realization mentioned above. In particular the spherical harmonics correspond to the traceless supersymmetric tensors. In the completely reducible case this is also the Cartan product inside $\otimes^kV$. In the other cases the traceless supersymmetric tensors turn out to still correspond to the representation generated by the vector of highest weight inside $\otimes^k V$ even though this representation is not irreducible. This can be seen as a generalized notion of Cartan product. The full tensor powers $\otimes ^k L^{m|2n}_{(1,0,\cdots,0)}$ were the object of study in \cite{MR1632811}.

In classical harmonic analysis, the orthogonal group $O(m)$ corresponds to the group of isometries of $\mR^m$ which stabilize the origin, hence $\mathfrak{so}(m)$ corresponds to the Killing vector fields which stabilize the origin. We use the definition in \cite{MR2434470} of the superalgebra of Killing vector fields on a Riemannian supermanifold. By doing so we obtain the Lie superalgebra $\osp$ as the algebra to consider on $\mR^{m|2n}$. Since the generators of $\mathfrak{sl}_2$ are $\osp$-invariant we obtain the dual pair $(\mathfrak{osp}(m|2n),\mathfrak{sl}_2)$. Because this dual pair solves the above-mentioned problems, it is the correct Howe dual pair for super harmonic analysis on $\mR^{m|2n}$. The bigger algebra in which $\osp$ and $\mathfrak{sl}_2$ are each other's centralizers is $\mathfrak{osp}(4n|2m)$.  A more general theory of Howe dual pairs with orthosymplectic algebras can be found in \cite{MR2028498, MR1893457, MR1272070, MR1827078}. 

We prove that the integration over the supersphere as introduced in \cite{DBE1, MR2539324} is orthosymplectically invariant, which yields a nice unique characterization in Theorem \ref{Pizuniekosp}. The explicit expression for an invariant integration over the supersphere is important for the generalization of field theories on the sphere to superspace, see e.g. \cite{MR1897209, MR2125586} and for the theory of invariant integration on supergroups, see e.g. \cite{MR2667819, CZ, MR2172158}. This integration can also naturally be defined over the supersphere manifold $\mS^{m-1|2n}$, which can be imbedded in flat superspace as will be done in forthcoming work. The existence and unicity of the integral in that approach can be deduced from the general theory in \cite{MR2667819}. 

In \cite{MR2395482} it was proved that the spaces of spherical harmonics $\cH_k=\cP_k\cap$Ker$\Delta$ of fixed degree $k$ on $\mR^{m|2n}$ are irreducible $\mathfrak{osp}(m|2n)$-modules if $m-2n\ge 2$. In Theorem \ref{irrH} we prove that this result is still valid as long as $m-2n\not\in-2\mN$ holds. This gives a multiplicity free irreducible direct sum decomposition for the supersymmetric tensor space $T(V)$ under the action of $\mathfrak{sl}_2\times \osp$, with $V=L^{m|2n}_{(1,0,\cdots,0)}$ the natural representation space for $\osp$. This multiplicity free decomposition has the additional property that each $\mathfrak{sl}_2$-representation is paired up with exactly one $\mathfrak{osp}(m|2n)$-representation. This implies that $(\osp,\mathfrak{sl}_2)$ is a Howe dual pair for $T(V)\cong \cP$ if $m-2n\not\in-2\mN$. The results of this realization of the Howe duality are summarized in Theorem \ref{AllHowe}. Also the $\mathfrak{osp}(m|2n)$-representations of spherical harmonics for the cases $m-2n\in-2\mN$ are studied. They are always indecomposable, but the irreducibility depends on the degree $k$ of the spherical harmonics $\cH_k$. The results are summarized in Theorem \ref{irrH}, Equation \eqref{hkLk1} and \eqref{hkLk2} and Theorem \ref{Lkrep}.

Using all these results we obtain polynomial realizations of the simple $\osp$-module $L^{m|2n}_{(k,0,\cdots,0)}$ with highest weight $(k,0,\cdots,0)$ for all values of $(m,n,k)\in\mN^3$. This completes the construction of $L^{m|2n}_{(k,0,\cdots,0)}$ as the tracefree supersymmetric part of tensor products of $L^{m|2n}_{(1,0,\cdots,0)}$ in \cite{MR0621253} by including the cases $m-2n\in-2\mN$ in Theorem \ref{Lkrep}. As a side result this gives the dimension of the representation $L^{m|2n}_{(k,0,\cdots,0)}$ and the decomposition into irreducible representations as an $\sosp$-representation. This information is nontrivial since the representations are almost all atypical, see e.g. \cite{MR051963}. We also obtain the branching rules for $L^{m|2n}_{(k,0,\cdots,0)}$ as an $\mathfrak{osp}(m-1|2n)$-module in Theorem \ref{branchingThm}.

The paper is organized as follows. First a short introduction to harmonic analysis on superspace is given. The algebra of Killing vector fields on $\mR^{m|2n}$ which stabilize the origin is calculated to be $\osp$. Next, we show that the introduction of the Howe dual pair $(\mathfrak{osp}(m|2n),\mathfrak{sl}_2)$ solves the three problems of the dual pair $(SO(m)\times Sp(2n),\mathfrak{sl}_2)$. In particular the unicity of the supersphere integral as an orthosymplectically invariant functional is proved. Then, the $\osp$-representations of spherical harmonics and polynomials on the supersphere are studied which leads to a realization of the Howe duality $\mathfrak{sl}_2\times\osp\subset\mathfrak{osp}(4n|2m)$. Finally, the theory of spherical harmonics is used to obtain new information on the irreducible representations $L^{m|2n}_{(k,0,\cdots,0)}$.

\section{Preliminaries}

\label{preliminaries}

\subsection{Harmonic analysis on Euclidean space}
\label{classHarm}
We consider the Euclidean space $\mR^m$ with $m$ variables $\ux=(x_1,\cdots,x_m)$. The standard orthogonal metric leads to the differential operators $\nabla^2_b=\sum_{j=1}^m\partial_{x_j}^2$, $r^2=\sum_{j=1}^mx_j^2$ and $\mE_b=\sum_{j=1}^mx_j\partial_{x_j}$ acting on functions on $\mR^m$. The operators $-\nabla^2_b/2$, $r^2/2$ and $\mE_b+m/2$ generate the Lie algebra $\mathfrak{sl}_2$. Since they are invariant under the action of the orthogonal group $O(m)$ we obtain the Howe dual pair $(\mathfrak{so}(m),\mathfrak{sl}_2)$, see \cite{MR0986027}. This duality is captured in the Fischer decomposition of the space of polynomials
\begin{eqnarray*}
\mR[x_1,\cdots,x_m]&=&\bigoplus_{j=0}^\infty\bigoplus_{k=0}^\infty r^{2j}\cH_k^b,
\end{eqnarray*}
with $\cH_k^b$ the polynomial null-solutions of $\nabla_b^2$ homogeneous of degree $k$. The blocks $r^{2j}\cH_k^b$ correspond to irreducible subrepresentations of the $\mathfrak{so}(m)$-representation on $\mR[x_1,\cdots,x_m]$. The isotypic components, which are the spaces $\mR[r^2]\cH_k^b=\bigoplus_{j=0}^\infty r^{2j}\cH_k^b$ are irreducible lowest weight modules for $\mathfrak{sl}_2$, with lowest weight $k+m/2$ and weight vectors $r^{2j}\cH_k^b$. This implies that $\mR[x_1,\cdots,x_m]$ has a multiplicity free irreducible direct sum decomposition under the action of $\mathfrak{sl}_2\times \mathfrak{so}(m)$.

The integral $\int_{\mS^{m-1}}$, represents integration over the unit the sphere and is the unique linear functional on $\cC^\infty(\mR^m)$ which is $\mathfrak{so}(m)$-invariant, satisfies $\int_{\mS^{m-1}}1=\sigma_m=\frac{2\pi^{m/2}}{\Gamma(m/2)}$ and 
\begin{eqnarray*}
\int_{\mS^{m-1}}(r^2-1)f=0&&\mbox{for all } f\in  \cC^\infty(\mR^m) .
\end{eqnarray*}
Note that we omit the measure $d\sigma$ on the sphere in this notation. The reason is that we see the unit sphere integration as an invariant functional on the space of functions on $\mR^m$ rather that integration over the manifold $\mS^{m-1}$ with an invariant measure. This approach will allow us to introduce the supersphere integral in a short efficient manner, without considering the details of the supersphere manifold.

\subsection{Harmonic analysis on superspace}
\label{subprel}

We repeat some results on the theory of harmonic analysis on the flat supermanifold $\mR^{m|2n}=(\mR^m,\cC^\infty_{\mR^m}\otimes\Lambda_{2n})$, as developed in \cite{MR2539324, CDBS3, DBE1}. The main object is the algebra of sections of the sheaf, which is the superalgebra ($\mZ_2$-graded algebra) of superfunctions $\cO(\mR^{m|2n})=\cC^\infty(\mR^m)\otimes\Lambda_{2n}$, where $\Lambda_{2n}$ is the Grassmann algebra generated by $2n$ anti-commuting variables, denoted by ${x\grave{}}_i$. The supervector $\bold{x}$ is defined as

\[
\bold{x}=(X_1,\cdots,X_{m+2n})=(\ux,\uxb)=(x_1,\cdots,x_m,{x\grave{}}_1,\cdots,{x\grave{}}_{2n}).
\]

The commutation relations for the Grassmann algebra and the bosonic variables are captured in the relation $X_iX_j=(-1)^{[i][j]}X_jX_i$ with $[i]=0$ if $i\le m$ and $[i]=1$ otherwise. Elements of $\cO(\mR^{m|2n})$ or $\Lambda_{2n}$ that consist of terms that are the product of an even amount of generators of $\Lambda_{2n}$ are called even and we use the notation $|f|=0$, for such functions. Odd functions are defined likewise and then $|f|=1$ holds.

The orthosymplectic metric $g$ on $\mR^{m|2n}$ is defined as $g\in\mR^{(m+2n)\times(m+2n)}$
\begin{eqnarray}
\label{defg}
g&=&\left( \begin{array}{c|c} I_m&0\\ \hline \vspace{-3.5mm} \\0&J
\end{array}
 \right)
\end{eqnarray}
with $J_{2n}\in\mR^{2n\times 2n}$ given by 
\begin{eqnarray}
\label{Jmatrix}
J_{2n}&=&\frac{1}{2}\left( \begin{array}{cccccc} 0&-1&&&\\1&0&&&\\&&\ddots&&\\&&&0&-1\\&&&1&0 
\end{array}
 \right).
 \end{eqnarray}

Define $X^j=\sum_iX_ig^{ij}$. The square of the `radial coordinate' is given by
\begin{eqnarray*}
R^2=\langle \bold{x},\bold{x}\rangle=\sum_{j=1}^{m+2n}X^jX_j =\sum_{i=1}^mx_i^2-\sum_{j=1}^n{x\grave{}}_{2j-1}{x\grave{}}_{2j}=r^2+\theta^2.
\end{eqnarray*}
The partial derivatives are denoted by $\nabla_j=\partial_{X^j}$. The raising of indices is given by $\nabla^j=\sum_i\nabla_ig^{ij}$, which implies $\nabla^j=(-1)^{[j]}\partial_{X_j}$. The super Laplace operator is given by 
\[\nabla^2=\sum_{j=1}^{m+2n}\nabla^j\nabla_j=\nabla^2_b-4\sum_{j=1}^n\partial_{{x\grave{}}_{2j-1}}\partial_{{x\grave{}}_{2j}}.\]
The super Euler operator is defined as $\mE=\sum_{i=1}^mx_i\partial_{x_i}+\sum_{j=1}^{2n}{x\grave{}}_j\partial_{{x\grave{}}_j}=\mE_b+\mE_f$. The operators $-\nabla^2/2$, $R^2/2$ and $\mE+M/2$, with $M=m-2n$, again generate the Lie algebra $\mathfrak{sl}_2$, see \cite{DBE1}:
\begin{eqnarray}
\nonumber
\left[\nabla^2/2,R^2/2\right]&=&\mE+M/2,\\
\label{sl2rel}
\left[\nabla^2/2,\mE+M/2\right]&=&2\nabla^2/2\quad \mbox{and}\\
\nonumber
\left[R^2/2,\mE+M/2\right]&=&-2R^2/2.
\end{eqnarray}
Since the parameter $M$ appears where in classical harmonic analysis the dimension appears, we call $M=m-2n$ the superdimension. The main new feature is that this value can be negative, which implies that the diagonal element of $\mathfrak{sl}_2$ can have negative eigenvalues in this realization.

The space of super polynomials is defined as
\begin{eqnarray*}
\cP=\mR[x_1,\cdots,x_m]\otimes \Lambda_{2n}&\subset& \cO(\mR^{m|2n})=\cC^\infty(\mR^m)\otimes\Lambda_{2n}.
\end{eqnarray*}
The homogeneous polynomials of degree $k$ are the elements $P\in\cP$ which satisfy $\mE P=kP$. The corresponding space is denoted by $\cP_k$. The null-solutions of the super Laplace operator are called harmonic superfunctions. The space of the spherical harmonics of degree $k$ is denoted by $\cH_k=\cP_k\cap\ker \nabla^2$. In the purely fermionic case $m=0$ we use the notation $\cH_{k}^{f}$. The Fischer decomposition holds in superspace when the superdimension is not even and negative or in the purely fermionic case, see \cite{DBE1}.

\begin{lemma}[Fischer decomposition]
If $M=m-2n \not \in -2 \mN$, $\cP$ decomposes as
\begin{eqnarray}
\cP = \bigoplus_{k=0}^{\infty} \cP_k= \bigoplus_{j=0}^{\infty} \bigoplus_{k=0}^{\infty} R^{2j}\cH_k.
\label{superFischer}
\end{eqnarray}
In case $m=0$, the decomposition is given by $\Lambda_{2n} = \bigoplus_{k=0}^{n} \left(\bigoplus_{j=0}^{n-k} \theta^{2j} \cH^f_k \right)$.
\label{superFischerLemma}
\end{lemma}

If $m\not=0$, the dimension of $\cH_k$ is given by
\begin{eqnarray}
\label{dimHk}
\dim\cH_k&=&\sum_{i=0}^{\min(k,2n)}\binom{2n}{i}\binom{k-i+m-1}{m-1}-\sum_{i=0}^{\min(k-2,2n)}\binom{2n}{i}\binom{k-i+m-3}{m-1},
\end{eqnarray}
see \cite{DBE1}.

The matrix Lie group $O(m)\times Sp(2n)$ consists of all matrices $S\in\mR^{(m+2n)\times(m+2n)}$ satisfying
\begin{eqnarray}
\label{OxSpdef}
\langle S\cdot\bold{x},S\cdot\bold{x}\rangle=\langle\bold{x},\bold{x}\rangle,
\end{eqnarray}
which is equivalent with $S^TgS=g$. Each such matrix $S$ is a block matrix $S=\left( \begin{array}{c|c} A&0\\ \hline \vspace{-3.5mm} \\0&B
\end{array}
 \right)$ with $A\in\mR^{m\times m}$ satisfying $A^TA=I_m$ and $B\in\mR^{2n\times 2n}$ satisfying $B^TJB=J$. The matrix $J$ is given in Equation \eqref{Jmatrix}. The action of $O(m)\times Sp(2n)$ on $\cC^\infty(\mR^m)\otimes\Lambda_{2n}$ is given by
\begin{eqnarray*}
\left(S,f(\bold{x})\right)\to f(S^{-1}\cdot\bold{x})
\end{eqnarray*}
such that $(S\cdot\bold{x})_j=\sum_{j=1}^{m+2n}S_{jk}X_k$.

Since $R^2$, $\nabla^2$ and $\mE+M/2$ generate $\mathfrak{sl}_2$ and are $O(m)\times Sp(2n)$-invariant, we obtain the dual pair $\left(\sosp,\mathfrak{sl}_2\right)$. This is again closely related to the Fischer decomposition \eqref{superFischer} for $M\not\in-2\mN$ or $m=0$. The blocks $\bigoplus_{j}R^{2j}\cH_k$ are irreducible lowest weight $\mathfrak{sl}_2$-representations with weight vectors $R^{2j}\cH_k$ and lowest weight $k+M/2$. The weight vectors $R^{2j}\cH_k$ are $SO(m)\times Sp(2n)$-representations. However, in full superspace ($m\not=0\not= n$), these representations are not irreducible. This also implies that $\cP$ does not correspond to a multiplicity free irreducible direct sum decomposition for $ \mathfrak{sl}_2\times (\sosp)$.

 In \cite{DBE1} the $\sosp$-module $\cH_k$ was decomposed into simple modules. First, the following polynomials need to be introduced.
\begin{lemma}
If $0\le q \le n$ and $0\le k\le n-q$, there exists a homogeneous polynomial $f_{k,p,q}=f_{k,p,q}(r^2,\theta^2)$ (unique up to a multiplicative constant) of degree $2k$ such that $f_{k,p,q} \cH_p^b \otimes \cH_q^f \neq 0$ and $\Delta (f_{k,p,q} \cH_p^b \otimes \cH_q^f) = 0$.
This polynomial is given explicitly by
\[
 f_{k,p,q}=\sum_{s=0}^ka_sr^{2k-2s}\theta^{2s} \quad \mbox{with}\quad a_s=\binom{k}{s}\frac{(n-q-s)!}{\Gamma (\frac{m}{2}+p+k-s)}\frac{\Gamma(\frac{m}{2}+p+k)}{(n-q-k)!}.
\]
\label{polythm}
\end{lemma}

In particular, we find $f_{0,p,q}=1$. Using these polynomials we can obtain a full decomposition of the space of spherical harmonics of degree $k$.

\begin{theorem}[Decomposition of $\cH_k$]
Under the action of $\sosp$ the space $\cH_k$ decomposes as
\label{decompintoirreps}
\[
\cH_{k} = \bigoplus_{j=0}^{\min(n, k)} \bigoplus_{l=0}^{\min(n-j,\lfloor \frac{k-j}{2} \rfloor)} f_{l,k-2l-j,j} \cH^b_{k-2l-j} \otimes \cH^f_{j},
\]with $f_{l,k-2l-j,j}$ the polynomials determined in Lemma \ref{polythm}.
\end{theorem}

The corresponding projection operators are given in \cite{DBE1},
\begin{equation}
\label{projpiecesSphHarm}
\mQ_{r,s}^k= \prod_{i=0, \;  i \neq k-2r-s}^{k} \dfrac{\Delta_{LB,b} + i(m-2+i)}{(i-k+2r+s)(k+i-2r-s+m-2)}\times  \prod_{j=0, \;  j \neq s}^{\min{(n,k)}} \dfrac{\Delta_{LB,f} + j(-2n-2+j)}{(j-s)(j+s-2n-2)},
\end{equation}
with $\Delta_{LB,b}$ and $\Delta_{LB,f}$ as defined in Equation \eqref{LB} or \eqref{LBosp} for the cases $n=0$ and $m=0$. They satisfy
\[
\mQ_{r,s}^k \left( f_{l,k-2l-j,j} \cH^b_{k-2l-j} \otimes \cH^f_{j} \right)= \delta_{rl} \delta_{sj} f_{l,k-2l-j,j} \cH^b_{k-2l-j} \otimes \cH^f_{j}.
\]

In case $m\not=0$, the supersphere is algebraically defined by the relation $R^2=1$. The integration over the supersphere was introduced in \cite{DBE1} for polynomials and generalized to a broader class of functions in \cite{MR2539324}. The uniqueness of this integral was also proved in \cite{DBE1, MR2539324}. 
\begin{theorem}
\label{SSintOxSp}
When $m\not=0$, the unique (up to a multiplicative constant) linear functional $T: \cP \rightarrow \mC$ satisfying the following properties for all $f(\bold{x}) \in \cP$:
\begin{itemize}
\item $T(R^2 f(\bold{x})) = T(f(\bold{x}))$
\item $T(f(S \cdot \bold{x})) = T(f(\bold{x}))$, \quad for all $ S \in SO(m)\times Sp(2n)$
\item $k \neq l \quad \Longrightarrow \quad T(\cH_k \cH_l) = 0 $
\end{itemize}
is given by the Pizzetti integral
\begin{eqnarray*}
\int_{\mS^{m-1|2n}} f(\bold{x})  =  \sum_{k=0}^{\infty}  \frac{2 \pi^{M/2}}{2^{2k} k!\Gamma(k+M/2)} (\nabla^{2k} f )(0)\quad \mbox{for}\quad f(\bold{x})\in\cP.
\end{eqnarray*}
\end{theorem}
In the purely bosonic case, the third condition is not necessary for uniqueness, contrary to the full superspace case.

The Berezin integral is the translation invariant linear functional on the Grassmann algebra,
\begin{eqnarray*}
\int_B&=&\pi^{-n} \partial_{{x \grave{}}_{2n}} \ldots \partial_{{x \grave{}}_{1}}.
\end{eqnarray*}

The orthosymplectic superalgebra $\mathfrak{osp}(m|2n)$ can be generated by the following differential operators on $\mR^{m|2n}$:
\begin{eqnarray}
\label{ospgen}
L_{ij}&=&X_i\partial_{X^j}-(-1)^{[i][j]}X_j\partial_{X^i}
\end{eqnarray}
for $1\le i \le j \le m+2n$, see \cite{CDBS3, MR2395482}. The super Lie bracket is realized by the graded commutator $[L_{ij},L_{kl}]=L_{ij}L_{kl}-(-1)^{([i]+[j])([k]+[l])}L_{kl}L_{ij}$. The Laplace-Beltrami operator is defined as
\begin{eqnarray}
\label{LB}
\Delta_{LB}&=&R^2\nabla^2-\mE(M-2+\mE).
\end{eqnarray}
This is a Casimir operator of $\mathfrak{osp}(m|2n)$ of degree $2$ and can be expressed as (see \cite{MR0546778})
\begin{eqnarray}
\label{LBosp}
\Delta_{LB}&=&-\frac{1}{2}\sum_{i,j,k,l=1}^{m+2n}L_{ij}g^{il}g^{jk}L_{kl}.
\end{eqnarray} 
This also corresponds to a quadratic Casimir operator for the $\mathfrak{sl}_2$-realization, as can be seen from equation \eqref{LB}.

\section{The orthosymplectic superalgebra as Killing vector fields}
\label{introOSp}

The pair $\left(\sosp,\mathfrak{sl}_2\right)$ on $\mR^{m|2n}$ introduced above does not satisfy the requirements to be an actual Howe dual pair. This can be expressed in the following remarks:
\begin{itemize}
\item \textbf{P1}: The weight vectors $R^{2j}\cH_k$ of $\mathfrak{sl}_2$ in the Fischer decomposition \eqref{superFischer} are not irreducible $\sosp$ representations, see Theorem \ref{decompintoirreps}. As a consequence the polynomials do not have a multiplicity free decomposition into irreducible pieces under the joint action $\mathfrak{sl}_2\times (\sosp)$.
\item \textbf{P2}: The only polynomials invariant under the action of the dual partner of $\mathfrak{sl}_2$ should be generated by the polynomial in the $\mathfrak{sl}_2$-algebra realization, $R^2$. However, all the elements of the commutative algebra generated by $r^2$ and $\theta^2$ are $\sosp$-invariant.
\item \textbf{P3}: The supersphere integration is not uniquely determined by the $\sosp$-invariance and the modulo $R^2-1$ property, see Theorem \ref{SSintOxSp}.
\end{itemize}
These problems will be solved by introducing the orthosymplectic superalgebra $\osp$. Since the Laplace operator and $R^2$ commute with the $\osp$-action this will prove that the pair $(\mathfrak{osp}(m|2n),\mathfrak{sl}_2)$ is a true Howe dual pair for harmonic analysis on $\mR^{m|2n}$. Problem \textbf{P2} was solved in Theorem 3 of \cite{CDBS3}, which implies that the only $\osp$-invariant polynomials on $\mR^{m|2n}$ are in fact given by Span$\{R^{2j}|j\in\mN\}$. The solutions to \textbf{P1} and \textbf{P3} are given in Theorem \ref{irrH} and Theorem \ref{Pizuniekosp} respectively. The mathematical motivation to consider the orthosymplectic superalgebra is given in the subsequent Theorem \ref{isomR}.

The orthogonal group is the matrix group which corresponds to the case $n=0$ in Equation \eqref{OxSpdef}. This group can also be characterized as the group of isometries of $\mR^m$ which stabilize the origin, therefore, the Lie algebra of Killing vector fields which stabilize the origin is isomorphic to $\mathfrak{so}(m)$. In \cite{MR2434470} the isometry supergroup of a Riemannian supermanifold was defined. We only use the Lie superalgebra component of this definition.

We consider the flat supermanifold $\mR^{m|2n}$ with space of vector fields $Der \cO(\mR^{m|2n})$ given by the left $\cO(\mR^{m|2n})$-module with basis $\{\partial_{X^1},\cdots,\partial_{X^{m+2n}}\}$. The metric $\langle \cdot,\cdot\rangle:Der \cO(\mR^{m|2n})\times Der \cO(\mR^{m|2n})\to\cO({\mR^{m|2n}})$ is defined by
\begin{eqnarray*}
\langle f \nabla_j |h\nabla^k\rangle =(-1)^{|h|[j]}\delta_j^k fh
\end{eqnarray*}
for $f,h\in\cO({\mR^{m|2n}})$ and the partial derivatives $\nabla_j$ and $\nabla^j$ as defined in Subsection \ref{subprel}. This also implies $\langle \nabla^l|\nabla^k\rangle=g^{kl}$, with $g$ the metric in Equation \eqref{defg}. The Lie superalgebra of graded Killing vector fields is generated by all homogeneous vector fields $F$ such that
\begin{eqnarray}
\label{defKilling}
F\langle Y|Z\rangle&=&\langle [F,Y]|Z\rangle+(-1)^{|F||Y|}\langle Y| [F,Z]\rangle
\end{eqnarray}
holds for all homogeneous vector fields $Y,Z\in Der \cO(\mR^{m|2n})$.

\begin{theorem}
\label{isomR}
The Killing vector fields on the Riemannian superspace $\mR^{m|2n}$ are given by the action of $\mathfrak{osp}(m|2n)$ on $\mR^{m|2n}$ in Equation \eqref{ospgen} and the partial derivatives $\partial_{X^j}$. The algebra of Killing vector fields which stabilize the origin is therefore $\osp$.
\end{theorem}

\begin{proof}
We consider condition \eqref{defKilling} for the Lie superalgebra in case $Y=\nabla^j$ and $Z=\nabla^k$ with $j,k=1,\cdots,m+2n$,
\begin{eqnarray*}
F\langle \nabla^j|\nabla^k\rangle=0=\langle [F,\nabla^j]|\nabla^k\rangle+(-1)^{|F|[j]}\langle\nabla^j| [F,\nabla^k]\rangle.
\end{eqnarray*}
For a homogeneous vector field $F=\sum_{l=1}^{m+2n}F^l\nabla_l$ of degree $|F|$, the relation $[F,\nabla^j]=-(-1)^{|F|[j]}\sum_l\nabla^j(F^l)\nabla_l$ shows this condition is equivalent with
\begin{eqnarray*}
\nabla^j(F^k)+(-1)^{([j]+[k])|F|}(-1)^{[j][k]}\nabla^k(F^j)&=&0.
\end{eqnarray*}
This leads to
\begin{eqnarray*}
(\nabla^i\nabla^jF^k)&=&-(-1)^{([j]+[k])|F|}(-1)^{[j][k]+[i][k]}(\nabla^k\nabla^iF^j).
\end{eqnarray*}
Applying this consecutively for $(i,j,k)$, $(k,i,j)$ and $(j,k,i)$ yields $(\nabla^i\nabla^jF^k)=-(\nabla^i\nabla^jF^k)$. This implies that the functions $F^k$ are polynomials of maximal degree $1$. Taking the functions $F^k$ constant leads to the partial derivatives $\partial_{X^j}$. When the functions $F^k$ are elements of the vectorspace $\cP_1=\mC\{X_j|j=1,\cdots,m+2n\}$, the Killing vector field is of the form $F=\sum_{kl}X_kF^{kl}\nabla_l$ for $F^{kl}\in\mC$, where the coefficients satisfy
\begin{eqnarray*}
(-1)^{[j]}F^{jk}+(-1)^{([j]+[k])|F|}(-1)^{[j][k]}(-1)^{[k]}F^{kj}&=&0,
\end{eqnarray*}
which can be reduced to 
\begin{eqnarray*}
F^{jk}+(-1)^{[j][k]}F^{kj}&=&0.
\end{eqnarray*}
These derivatives corresponds to the realization of $\mathfrak{osp}(m|2n)$ in Equation \eqref{ospgen}. It is straightforward to calculate that these $F$ also satisfy the relation $F\langle Y|Z\rangle=\langle [F,Y]|Z\rangle+(-1)^{|F||Y|}\langle Y| [F,Z]\rangle$ for general vector fields $Y$ and $Z$.
\end{proof}

\section{The supersphere integral}

The object of the supersphere manifold is related to the equation $R^2=1$ on $\mR^{m|2n}$. In this paper we will not need the supersphere manifold explicitly, only the algebra of functions. The algebra of functions on the supersphere is given by $\cO(\mR^{m|2n})/(R^2-1)$ with $(R^2-1)$ the ideal generated by the function $R^2-1$. The supersphere integral therefore has to be a functional on this function space, which is equivalent with a functional on $\cO(\mR^{m|2n})$ such that the ideal $(R^2-1)$ is in the kernel. In the following theorem we prove the uniqueness of the supersphere integration for functions on $\mR^{m|2n}$, thus solving problem \textbf{P3}. This gives a natural generalization of the characterization of the integral over the unit sphere in Subsection \ref{classHarm}. Contrary to Theorem \ref{SSintOxSp} we consider the functional on general smooth functions instead of on polynomials. When restricted to polynomials the integration will be identical to the Pizzetti formula in Theorem \ref{SSintOxSp}.
\begin{theorem}
\label{Pizuniekosp}
When $m\not=0$, the only (up to a multiplicative constant) linear functional $T: \cO(\mR^{m|2n}) \rightarrow \mC$, satisfying the properties
\begin{itemize}
\item $T[R^2 f]=\, T[f]$, 
\item $T$ is $\osp$-invariant,
\end{itemize}
is given by
\begin{eqnarray*}
\int_{\mS^{m-1|2n}}\cdot&=&\int_{\mS^{m-1}}\int_B \left(1-\theta^2\right)^{\frac{m}{2}-1}\phi^\sharp\cdot,
\end{eqnarray*}
with $\phi^\sharp$ the superalgebra morphism given by $\phi^\sharp:\cC^\infty(\mR^{m})\otimes\Lambda_{2n}\to\cC^\infty(\mR^{m}_0)\otimes\Lambda_{2n}$ (with $\mR^m_0=\mR^m\backslash \{0\}$) as
\begin{eqnarray*}
\phi^\sharp(f(\bold{x}))&=&\sum_{j=0}^n\frac{(-1)^j\theta^{2j}}{j!}\left(\frac{\partial}{\partial r^2}\right)^jf(\bold{x}),
\end{eqnarray*}
with $\partial_{r^2}=\frac{1}{2r^2}\mE_b$.
\end{theorem}

\begin{proof}
If $[\phi^\sharp(f)]_{r=1}=0$ then $\phi^\sharp(f)=(r^2-1)h$ for some $h\in \cO(\mR^{m|2n}_0)$. The morphism $\phi^\sharp$ is invertible, $(\phi^\sharp)^{-1}=\sum_{j=0}^n\frac{\theta^{2j}}{j!}(\partial_{r^2})^j$. Therefore 
\[f=\left(\left(\phi^\sharp\right)^{-1}(r^2-1)\right)\left(\left(\phi^\sharp\right)^{-1}(h)\right)=(R^2-1)\left(\left(\phi^\sharp\right)^{-1}(h)\right)\]
and $T[f]=0$. This implies that $T[f]$ only depends on $\left[\phi^\sharp (f)\right]_{r=1}$. 

Since $T$ is also $\mathfrak{so}(m)\oplus\mathfrak{sp}(2n)$-invariant $T[f]$ must be of the form
\begin{eqnarray*}
T[f]&=&\int_{\mS^{m-1}}\int_B \alpha(\theta^2)\phi^\sharp(f),
\end{eqnarray*}
for $\alpha(\theta^2)$ some polynomial in $\theta^2$, since $\int_B\theta^{2i}\cdot$ for $i=1,\cdots,n$ are the only $\mathfrak{sp}(2n)$-invariant linear functionals on $\Lambda_{2n}$.
Now we demand that
\begin{eqnarray*}
T\circ L_{ij}&=&0
\end{eqnarray*}
holds for $L_{ij}$ the odd $\osp$-generators in equation \eqref{ospgen}. The subsequent Lemma \ref{OO} implies
\begin{eqnarray*}
& &\phi^\sharp \circ L_{i,m+j}\\
&=&\left[x_i\sqrt{1-\frac{\theta^2}{r^2}}\partial_{{x\grave{}}^j} +2x_i\sqrt{1-\frac{\theta^2}{r^2}}\frac{{x\grave{}}_j}{2r}\partial_{r}\ -{x\grave{}}_j\frac{1}{\sqrt{1-\frac{\theta^2}{r^2}}}\partial_{x_i} +{x\grave{}}_j\frac{x_i\theta^2}{r^3\sqrt{1-\frac{\theta^2}{r^2}}}\partial_{r}\right]\circ\phi^\sharp \\
&=&x_i\sqrt{1-\frac{\theta^2}{r^2}}\partial_{{x\grave{}}^j}\circ\phi^\sharp -{x\grave{}}_j\frac{1}{\sqrt{1-\frac{\theta^2}{r^2}}}\partial_{x_i}\circ\phi^\sharp +\frac{{x\grave{}}_jx_i}{r\sqrt{1-\frac{\theta^2}{r^2}}}\partial_{r}\circ\phi^\sharp \\
&=&x_i\sqrt{1-\frac{\theta^2}{r^2}}\partial_{{x\grave{}}^j}\circ\phi^\sharp -\frac{{x\grave{}}_j}{r^2\sqrt{1-\frac{\theta^2}{r^2}}}\sum_{k=2}^mx_kL_{ki}\circ\phi^\sharp.
\end{eqnarray*}
Applying this for $L_{1,m+2}=-2x_1\partial_{{x\grave{}}_1}-{x\grave{}}_2\partial_{x_1}$ in the obtained expression for $T$ and using $\int_B\partial_{{x\grave{}}_1} =0$ yields
\begin{eqnarray*}
& &\int_{\mS^{m-1}}\int_B \alpha(\theta^2)\phi^\sharp L_{1,m+2}\cdot\\
&=&-\int_{\mS^{m-1}}x_1\int_B \alpha(\theta^2)\left[2\sqrt{1-\frac{\theta^2}{r^2}}\partial_{{x\grave{}}_1} +(m-1)\frac{{x\grave{}}_2}{r^2\sqrt{1-\frac{\theta^2}{r^2}}}\right]\phi^\sharp\cdot \\
&=&\int_{\mS^{m-1}}x_1\int_B \left[\left(\partial_{{x\grave{}}_1}\alpha(\theta^2)2\sqrt{1-\frac{\theta^2}{r^2}}\right)-(m-1)\frac{{x\grave{}}_2\alpha(\theta^2)}{r^2\sqrt{1-\frac{\theta^2}{r^2}}}\right]\phi^\sharp\cdot.
\end{eqnarray*}
Since $\phi^\sharp$ is invertible, the relation
\begin{eqnarray*}
\partial_{{x\grave{}}_1}\left(\alpha(\theta^2)\sqrt{1-\theta^2}\right)&=&\frac{m-1}{2}\frac{{x\grave{}}_2}{\sqrt{1-\theta^2}}\alpha(\theta^2)
\end{eqnarray*}
must hold in order for the operator above to be identically zero. A straightforward calculation shows that this uniquely determines $\alpha(\theta^2)$ to be $(1-\theta^2)^{\frac{m}{2}-1}$ up to a multiplicative constant. It is clear that this $T$ will also be zero when composed with the other $L_{ij}\in\mathfrak{osp}(m|2n)_1$.
\end{proof}

When the supersphere integral is restricted to polynomials, we re-obtain the Pizzetti formula from Theorem \ref{SSintOxSp}, as has been proved in \cite{MR2539324}.

Now we derive the properties of the morphism $\phi^\sharp$ used in the proof of Theorem \ref{Pizuniekosp}.
\begin{lemma}
\label{OO}
The morphism $\phi^\sharp:\cC^\infty(\mR^m)\otimes\Lambda_{2n}\to\cC^\infty(\mR^m_0)\otimes\Lambda_{2n}$ behaves with respect to the coordinates and derivatives as
\begin{eqnarray*}
\phi^\sharp x_j=x_j\sqrt{1-\frac{\theta^2}{r^2}}\phi^\sharp,& &\phi^\sharp{x\grave{}}_j={x\grave{}}_j\phi^\sharp\\
\phi^\sharp\partial_{x_j}=\frac{1}{\sqrt{1-\frac{\theta^2}{r^2}}}\partial_{x_j}\phi^\sharp-\frac{x_j\theta^2}{r^3\sqrt{1-\frac{\theta^2}{r^2}}}\partial_{r}\phi^\sharp,& &\phi^\sharp\partial_{{x\grave{}}_{j}}=\partial_{{x\grave{}}_{j}}\phi^\sharp+\frac{\sum_{i=1}^{2n}J_{ji}{x\grave{}}_{i}}{r}\partial_{r}\phi^\sharp,
\end{eqnarray*}
with $J_{ij}$ given in equation \eqref{Jmatrix}.
\end{lemma}
\begin{proof}
Since $\phi^\sharp$ is a superalgebra morphism, $\phi^\sharp( X_j f(\bold{x}))=\phi^\sharp (X_j)\phi^\sharp(f(\bold{x}))$ holds. The calculation of the expressions $\phi^\sharp (x_j)$ and $\phi^\sharp ({x\grave{}}_j)$ are straightforward. To calculate the third property we use the fact that $\partial_{r^2}$ and $L_{ij}$ for $1\le i,j\le m$ commute with $\phi^\sharp$ and the first property in the lemma, leading to
\begin{eqnarray*}
\partial_{x_j}\phi^\sharp&=&\left[\frac{x_j}{r}\partial_r+\sum_{l=1}^m\frac{x_l}{r^2}L_{lj}\right]\phi^\sharp\\
&=&\frac{x_j}{r}\partial_r\phi^\sharp+\sum_{l=1}^m\sqrt{1-\frac{\theta^2}{r^2}}\,\phi^\sharp \,\frac{x_l}{r^2}L_{lj}\\
&=&\frac{x_j}{r}\partial_r\phi^\sharp+\sqrt{1-\frac{\theta^2}{r^2}}\phi^\sharp \,\left[\partial_{x_j}-\frac{x_j}{r}\partial_r\right]\\
&=&\frac{x_j}{r}\left(1-(1-\frac{\theta^2}{r^2})\right)\partial_r\phi^\sharp+\sqrt{1-\frac{\theta^2}{r^2}}\phi^\sharp \,\partial_{x_j}.\\
\end{eqnarray*}
The last property follows from a straightforward calculation.
\end{proof}

\section{Polynomials on the supersphere}
\label{SectionPolySS}

\subsection{Polynomials on super Euclidean space and on the supersphere}

The polynomials on $\mR^{m|2n}$ form a $\mathfrak{gl}(m|2n)$-module, the action of $\mathfrak{gl}(m|2n)$ is given by the differential operators $X_i\partial_{X_j}$ for $i,j=1\cdots,m+2n$. The decomposition $\cP=\bigoplus_{k=0}^\infty\cP_k$ is the decomposition of the polynomials on $\mR^{m|2n}$ into irreducible pieces under the action of $\mathfrak{gl}(m|2n)$, where $\cP_k$ is the representation with highest weight $(k,0,\cdots,0)$. The $m|2n$-dimensional super vectorspace $V\cong\cP_1$ is also the natural representation space for $\mathfrak{gl}(m|2n)$, in fact $\mathfrak{gl}(m|2n)\cong$ End$(V)$. Thus we obtain that $S(V)_k= \odot^k V$ is an irreducible $\mathfrak{gl}(m|2n)$-representation.

In what follows we prove the corresponding results for $\osp$. So on the one hand we decompose $\cP_k\cong S(V)_k= \odot^k V$ into irreducible $\osp$-representations and on the other hand we study the polynomials on the supersphere $\mS^{m-1|2n}$ as an $\mathfrak{osp}(m|2n)$-representation. If $m-2n\not\in-2\mN$, the polynomials on the supersphere $\mS^{m-1|2n}$ can be identified with the harmonic polynomials on $\mR^{m|2n}$. This immediately follows from the Fischer decomposition \eqref{superFischer} and the fact that the functions on the supersphere correspond to $\cO(\mR^{m|2n})/(R^2-1)$, with $(R^2-1)$ the ideal generated by the function $R^2-1$.

The action of $\mathfrak{osp}(m|2n)$ on polynomials is given by the expressions in Equation \eqref{ospgen}. Since the $L_{ij}$ commute with $\mE$ and $\nabla^2$, the spaces $\cH_k$ are $\mathfrak{osp}(m|2n)$-representations. In Appendix A of \cite{MR2395482} the irreducibility of $\cH_k$ as an $\mathfrak{osp}(m|2n)$-representation was proved for $M=m-2n>1$. Some very specific examples were also proved in Proposition $3.1$ in \cite{MR2441598}. In the next section we will generalize this result to all $M\not\in-2\mN$, which implies that 
\[\cH=\bigoplus_{k=0}^\infty\cH_k\] is the decomposition of $\cH$ (or the polynomials on the supersphere) into irreducible blocks under the action of $\mathfrak{osp}(m|2n)$. This will also solve problem $\textbf{P1}$ and complete the interpretation of the Fischer decomposition \eqref{superFischer}. We will also consider the case $M\in-2\mN$ which leads to unexpected and interesting results.

The supersphere integration in Theorem \ref{Pizuniekosp} or Theorem \ref{SSintOxSp} also generates a superhermitian form $\langle\cdot|\cdot\rangle_{\mS^{m-1|2n}}$ on each space $\cH_k$, (or on $\cH$), by defining
\begin{eqnarray*}
\langle f|g\rangle_{\mS^{m-1|2n}}&=&\int_{\mS^{m-1|2n}}f\overline{g}.
\end{eqnarray*}
This implies $\langle f|g\rangle_{\mS^{m-1|2n}}=(-1)^{|f||g|}\overline{\langle g|f\rangle}_{\mS^{m-1|2n}}$, for $f$ and $g$ homogeneous.

The representation of $\mathfrak{osp}(m|2n)$ on $\cH$ is an orthosymplectic representation (see definition in \cite{MR0922822}) with respect to this hermitian form. For $f,g\in\cH_k$ with $f$ homogeneous,
\begin{eqnarray*}
\langle L_{ij}f|g\rangle_{\mS^{m-1|2n}}&=&\int_{\mS^{m-1|2n}}L_{ij}\left(f\overline{g}\right)-(-1)^{([i]+[j])|f|}\int_{\mS^{m-1|2n}}fL_{ij}\overline{g}\\
&=&-(-1)^{([i]+[j])|f|}\langle f|L_{ij}g\rangle_{\mS^{m-1|2n}}.
\end{eqnarray*} 

\begin{remark}
All the results on representations of the orthosymplectic Lie superalgebra in this paper can be extended to the orthosymplectic Lie supergroup. This is done most elegantly in the approach using supergroup pairs, the Lie supergroup ${OSp}(m|2n)$ then corresponds to the supergroup pair $(O(m)\times Sp(2n),\osp)$. More information on the definition of representations of such supergroup pairs is given in \cite{MR2207328}. The Lie supergroup ${OSp}(m|2n)$ can be introduced as the supergroup of isometries of $\mR^{m|2n}$ that stabilize the origin, according to the definition in \cite{MR2434470}.
\end{remark}

\subsection{The $\cH_k$-representations}

The main result of this subsection is given in the following theorem.
\begin{theorem}
\label{irrH}
When $M=m-2n\not\in-2\mN$, the space $\cH_k$ of spherical harmonics on $\mR^{m|2n}$ of homogeneous degree $k$ is an irreducible $\mathfrak{osp}(m|2n)$-module. When $M\in-2\mN$, $\cH_k$ is irreducible if and only if 
\begin{eqnarray*}
k>2-M &\mbox{or}& k< 2-\frac{M}{2}.
\end{eqnarray*}
The representation $\cH_k$ is always indecomposable.
\end{theorem}
The highest weights of the irreducible representations are given in equations \eqref{hkLk1} and \eqref{hkLk2}, the composition series for the reducible but indecomposable modules is given in Theorem \ref{Lkrep}. The remainder of this subsection is devoted to proving Theorem \ref{irrH}. This solves problem \textbf{P1}. 

\begin{remark}
Contrary to finite dimensional Lie algebras, finite dimensional Lie superalgebras do not possess the property that every finite dimensional representation is completely reducible. The representations $\cH_k$ for $M\in-2\mN$ and $2-\frac{M}{2}\le k\le 2-M$ are examples of representations of $\mathfrak{osp}(m|2n)$ which are indecomposable but not irreducible, hence not completely reducible.
\end{remark}

First we need the following technical lemma.
\begin{lemma}
The functions introduced in Lemma \ref{polythm} satisfy the relation
\label{Lf}
\begin{eqnarray*}
L_{i,2j-1+m}f_{k,p,q}&=&2k\left(\frac{M}{2}+p+q+k-1\right)f_{k-1,p+1,q+1}\, x_i\, {x\grave{}}_{2j-1}
\end{eqnarray*}
for $i\le m$ and $j\le n$ and $L_{il}$ defined in Equation \eqref{ospgen}.
\end{lemma}
\begin{proof}
This expression follows from a direct calculation. Since $L_{i,2j-1+m}=\left(2x_i\partial_{{x\grave{}}_{2j}}-{x\grave{}}_{2j-1}\partial_{x_i}\right)$, the relation
\begin{eqnarray*}
L_{i,2j-1+m}f_{k,p,q}
&=&2x_i{x\grave{}}_{2j-1}\sum_{s=1}^k\frac{k!}{(k-s)!(s-1)!}\frac{(n-q-s)!}{\Gamma (\frac{m}{2}+p+k-s)}\frac{\Gamma(\frac{m}{2}+p+k)}{(n-q-k)!}r^{2k-2s}\theta^{2s-2}\\
&-&{x\grave{}}_{2j-1}2x_i\sum_{s=0}^{k-1}\frac{k!}{(k-s-1)!s!}\frac{(n-q-s)!}{\Gamma (\frac{m}{2}+p+k-s)}\frac{\Gamma(\frac{m}{2}+p+k)}{(n-q-k)!}r^{2k-2s-2}\theta^{2s}\\
&=&2k(\frac{M}{2}+p+q+k-1)x_i{x\grave{}}_{2j-1}f_{k-1,p+1,q+1}
\end{eqnarray*}
holds, which proves the lemma.
\end{proof}

The decomposition of $\cH_k$ in Theorem \ref{decompintoirreps} can be represented by the following diagram, where we use the notation $(l,k-2l-j,j)\leftrightarrow f_{l,k-2l-j,j}\cH_{k-2l-j}^b\otimes\cH_j^f$ for the irreducible $\mathfrak{so}(m)\oplus\mathfrak{sp}(2n)$-representations.
\[
\xymatrix@=10pt@!C{
(0,k,0)&(0,k-1,1)&(0,k-2,2)&(0,k-3,3)&\cdots\\
&(1,k-2,0)&(1,k-3,1)\ar[u]_{\textbf{(1)}}\ar[d]_{\textbf{(4)}}\ar[r]_{\textbf{(3)}}\ar[l]_{\textbf{(2)}}&(1,k-4,2)&\cdots\\
&&(2,k-4,0)&(2,k-5,1)&\cdots\\
}
\]
The arrows represent the actions of the elements of $\cU(\mathfrak{osp}(m|2n))$ constructed in the following lemma, where we will use the notation $(l,p,q)\leftrightarrow f_{l,p,q}\cH_{p}^b\cH_q^f$ again. It is important for the sequel to note that the decomposition above is finite, which is clear from the fact that $\cH_k$ is finite dimensional.

\begin{lemma}
\label{arrows}
Consider the spherical harmonics on $\mR^{m|2n}$ with $m\not=0\not= n$. For $j=1,\cdots,4$, there are spherical harmonics $e_j\in (l,p,q)\not= 0$ (so $l+q\le n$ and $p\le 1$ when $m=1$) and elements $u_j\in \cU(\mathfrak{osp}(m|2n))$ such that $u_je_j$ is non-zero and
\[
u_1 e_1\in(l-1,p+1,q+1)\qquad \mbox{if} \qquad p+q+l\not=1-\frac{M}{2},
\]
\[
u_2 e_2\in (l,p+1,q-1), \quad u_3 e_3\in (l,p-1,q+1) \quad \mbox{and}\quad  u_4 e_4\in (l+1,p-1,q-1)\]
if the resulting subspace of spherical harmonics is non-zero.
\end{lemma}
\begin{proof}
The Laplace-Beltrami operators $\Delta_{LB,b}$ and $\Delta_{LB,f}$ are quadratic Casimir operators of respectively $\mathfrak{so}(m)$ and $\mathfrak{sp}(2n)$, see Equation \eqref{LBosp} for the limit cases $n=0$ and $m=0$. In particular they are elements of $\cU(\osp)$. This means the projection operators in Equation \eqref{projpiecesSphHarm} are elements of $\cU(\mathfrak{osp}(m|2n))$. 

First, consider, for $q<n$, a spherical harmonic $H_q^f\in\cH_q^f$ which does not contain ${x\grave{}}_1$ or ${x\grave{}}_2$. This exists since it can be taken as a spherical harmonic of degree $q$ in the Grassmann algebra $\Lambda_{2n-2}$ generated by the $\{{x\grave{}}_j,j>2\}$ (and $q\le n-1$). This also implies that ${x\grave{}}_1H_q^f\in\cH_{q+1}^f$ and is different from zero. Next, consider a spherical harmonic $H_p^b\in \cH_p^b$ (with $p\le 1$ when $m=1)$. Define $H_{p+1}^b\in \cH_{p+1}^b$ by
\begin{eqnarray*}
H_{p+1}^b&=&x_1H_p^b-\frac{r^2}{2p+m-2}\partial_{x_1}H_p^b.
\end{eqnarray*}
Lemma \ref{Lf} then implies
\begin{eqnarray*}
L_{1,1+m}\left(f_{l,p,q}H_p^bH_q^f\right)&=&2l(\frac{M}{2}+p+q+l-1)\,f_{l-1,p+1,q+1}\, H_{p+1}^b\, {x\grave{}}_{1}H_q^f\\
&+&\frac{l}{p+\frac{m}{2}-1}(\frac{M}{2}+p+q+l-1)\,r^2\,f_{l-1,p+1,q+1}\,\partial_{x_1}H_{p}^b\, {x\grave{}}_{1}H_q^f\\
&-&f_{l,p,q}\,\partial_{x_1}H_{p}^b\, {x\grave{}}_{1}H_q^f.
\end{eqnarray*}
The sum of the last two lines is of the form $g(r^2,\theta^2)\partial_{x_1}H_{p}^b{x\grave{}}_{1}H_q^f$ and is an element of $\cH_{2l+p+q}$ since the left-hand and right-hand side of the first line are. Lemma \ref{polythm} then yields $g(r^2,\theta^2)=\lambda_1 f_{l,p-1,q+1}$ for some $\lambda_{1}\in\mR$, with $\lambda_1\not=0$ since otherwise $f_{l,p,q}\equiv 0$ mod $r^2$. The identity $\frac{l}{p+\frac{m}{2}-1}(\frac{M}{2}+p+q+l-1)r^2f_{l-1,p+1,q+1}-f_{l,p,q}=\lambda_1 f_{l,p-1,q+1}$ also follows from a direct calculation, showing that $\lambda_1=-(1+\frac{l}{p+\frac{m}{2}-1})$. Summarizing, we obtain
\begin{eqnarray*}
L_{1,1+m}\left(f_{l,p,q}H_p^bH_q^f\right)&=&2l(\frac{M}{2}+p+q+l-1)f_{l-1,p+1,q+1}H_{p+1}^b{x\grave{}}_{1}H_q^f\\
&+&\lambda_1 f_{l,p-1,q+1}(\partial_{x_1}H_{p}^b)({x\grave{}}_{1}H_q^f)
\end{eqnarray*}
for some $\lambda_1\not=0$. Using this we can prove the arrows \textbf{(1)} and \textbf{(3)}. Both arrows start from spaces with $q<n$ since otherwise $\cH_{q+1}^f=0$.

\textbf{(1)}: Take $H_p^b$ and $H_q^f$ as defined above and assume $x_1H_p^b\not=\frac{r^2}{2p}\partial_{x_1}H_p^b$. This is always possible, the only non-trivial case is if $m=1$, but then $p$ has to be zero since otherwise $\cH_{p+1}^b=0$. The previous calculations then imply
\begin{eqnarray*}
\mQ^{2l+p+q}_{l-1,q+1}\left(L_{1,1+m}f_{l,p,q}H_p^bH_q^f\right)&=&2l(\frac{M}{2}+p+q+l-1)f_{l-1,p+1,q+1}H_{p+1}^b{x\grave{}}_{1}H_q^f\\
&\in& f_{l-1,p+1,q+1}\cH_{p+1}^b\otimes\cH_{q+1}^f
\end{eqnarray*}
which is different from zero since $l>0$ (otherwise $f_{l-1,p+1,q+1}\cH_{p+1}^b\otimes\cH_{q+1}^f=0$) and since we assumed $p+q+l\not=1-\frac{M}{2}$ for this arrow.

\textbf{(3)}: Take again $H_p^b$ and $H_q^f$ as defined above but now assume $\partial_{x_1}H_p^b\not=0$ (which is always possible since $p>0$ for this arrow). We then obtain
\begin{eqnarray*}
\mQ^{2l+p+q}_{l,q+1}\left(L_{1,1+m}f_{l,p,q}H_p^bH_q^f\right)&=&\lambda_1 f_{l,p-1,q+1}(\partial_{x_1}H_{p}^b)({x\grave{}}_{1}H_q^f)\\
&\in& f_{l,p-1,q+1}\cH_{p-1}^b\otimes\cH_{q+1}^f,
\end{eqnarray*}
which is different from zero.

The arrows \textbf{(2)} and \textbf{(4)} can be proved similarly. Consider for $1\le q\le n$ a spherical harmonic $\cH_{q-1}^f\in \cH_{q-1}^f$ which does not contain ${x\grave{}}_{1}$ or  ${x\grave{}}_{2}$ (and therefore ${x\grave{}}_{2}H_{q-1}^f\in\cH_q^f$), define $H_{q+1}^f\in\cH_{q+1}^f$ (not necessarily different from zero) by
\begin{eqnarray*}
H_{q+1}^f&=&{x\grave{}}_{1}{x\grave{}}_{2}H_{q-1}^f-\frac{1}{q-n-1}\theta^2H_{q-1}^f.
\end{eqnarray*}
Then we calculate
\begin{eqnarray*}
L_{1,1+m}\left(f_{l,p,q}H_p^b{x\grave{}}_2H_{q-1}^f\right)&=&2l(\frac{M}{2}+p+q+l-1)f_{l-1,p+1,q+1}x_1H_p^b{x\grave{}}_{1}{x\grave{}}_{2}H_{q-1}^f\\
&+&2f_{l,p,q}x_1H_p^bH_{q-1}^f-f_{l,p,q}\partial_{x_1}H_p^b{x\grave{}}_{1}{x\grave{}}_{2}H_{q-1}^f.
\end{eqnarray*}

This leads to
\begin{eqnarray*}
\mQ^{2l+p+q}_{l,q-1}\left(L_{1,1+m}f_{l,p,q}H_p^b{x\grave{}}_2H_{q-1}^f\right)=\lambda_2 f_{l,p+1,q-1}H_{p+1}^bH_{q-1}^f\in f_{l,p+1,q-1}\cH_{p+1}^b\cH_{q-1}^f.
\end{eqnarray*}
with $\lambda_2 f_{l,p+1,q-1}=2f_{l,p,q}-\frac{2l}{q-n-1}(\frac{M}{2}+p+q+l-1)\theta^2f_{l-1,p+1,q+1}$ with $\lambda_2\not=0$ since otherwise $f_{l,p,q}\equiv 0$ mod $\theta^2$.

The calculation above also implies
\begin{eqnarray*}
& &\mQ^{2l+p+q}_{l+1,q-1}\left(L_{1,1+m}f_{l,p,q}H_p^b{x\grave{}}_2H_{q-1}^f\right)\\
&=&\lambda_3 f_{l+1,p-1,q-1}(\partial_{x_1}H_{p}^b)H_{q-1}^f\in f_{l+1,p-1,q-1}\cH_{p-1}^b\cH_{q-1}^f,
\end{eqnarray*}
with
\begin{eqnarray*}
\lambda_3 f_{l+1,p-1,q-1}&=&\frac{l(\frac{M}{2}+p+q+l-1)}{(p+\frac{m}{2}-1)(q-n-1)}f_{l-1,p+1,q+1}\theta^2r^2+\frac{r^2}{p+\frac{m}{2}-1}f_{l,p,q}-\frac{\theta^2}{q-n-1}f_{l,p,q}.
\end{eqnarray*}

Again, $\lambda_3\not=0$, since otherwise $r^2f_{l,p,q}\equiv 0$ mod $\theta^2$. This implies arrows \textbf{(2)} and \textbf{(4)} by identifying $e_2=\mQ^{2l+p+q}_{l,q-1}\circ L_{1,1+m}$ and $e_4=\mQ^{2l+p+q}_{l+1,q-1}\circ L_{1,1+m}$.
\end{proof}

\begin{remark}
\label{irrEnv}
A  finite dimensional representation $V$ of a Lie superalgebra $\mathfrak{g}$ is irreducible if and only if for any two vectors $u,v\in V$ with $u\not=0$, there is an element $X$ in the universal enveloping algebra $\cU(\mathfrak{g})$ such that $Xu=v$. If a finite dimensional representation $V$ of a Lie superalgebra $\mathfrak{g}$ has one vector $u\in V$ such that for each vector $v\in V$ there is an element $X\in\cU(\mathfrak{g})$ that satisfies $Xu=v$ the representation is indecomposable.
\end{remark}

Now we are ready to prove Theorem \ref{irrH}.
\begin{proof}
For $M\in-2\mN$, the $\mathfrak{sl}_2$ relations for $\nabla^2$, $R^2$ and $\mE+\frac{M}{2}$, see \eqref{sl2rel} (or Lemma 3 in \cite{DBE1}), imply that for $2-\frac{M}{2}\le k\le 2-M$
\begin{eqnarray}
\label{submodule}
R^{2k+M-2}\cH_{2-M-k}\subset \cH_k.
\end{eqnarray}
The fact that $\Delta R^{2k+M-2}\cH_{2-M-k}=0$ can also be obtained from equation \eqref{LB} since $\Delta_{LB}$ commutes with $R^2$. Equation \eqref{submodule} defines a non-trivial submodule (since $\dim\cH_{2-M-k}< \dim \cH_{k}$) which proves the reducibility for $2-\frac{M}{2}\le k\le 2-M $. 

We prove the rest of the theorem using Remark \ref{irrEnv}. First we assume $m\not=1$ and take one fixed $H_k^b\in\cH_k^b\subset\cH_k$. We prove that for every element of the form
\begin{eqnarray*}
u&=&f_{l,k-2l-j}H_{k-2l-j}^bH_j^f\in f_{l,k-2l-j}\cH_{k-2l-j}^b\otimes\cH_j^f\\
\end{eqnarray*}
there is an elements $X\in \cU(\mathfrak{osp}(m|2n))$ for which $XH_k^b=u$. Since $\cH_k^b$ is an irreducible $\mathfrak{so}(m)$-representation there is an element $X_1\in \cU(\mathfrak{so}(m))$ such that $X_1H_k^b$ is an element in $\cH_k^b$ for which there is an element $X_2$ of $\cU(\mathfrak{osp}(m|2n))$ determined in Lemma \ref{arrows} (arrow \textbf{(3)}) such that
\begin{eqnarray*}
X_2X_1H_k^b&=&H_{k-1}^bH_1^f\in\cH_{k-1}^b\otimes\cH_1^f.
\end{eqnarray*}
This procedure can be repeated until we find an $X'\in \cU(\mathfrak{osp}(m|2n))$ and an $H_{k-j}^bH_j^f\in\cH_{k-j}^b\otimes\cH_j^f$ such that
\begin{eqnarray*}
X'H_k^b&=&H_{k-j}^bH_j^f.
\end{eqnarray*}
By the same arguments and by using arrow \textbf{(4)} in Lemma \ref{arrows} we find an $X\in \cU(\mathfrak{osp}(m|2n))$ such that $XH_k^b=u$. This shows the representation is always indecomposable.

For the cases $M\not\in-2\mN$ or $M\in-2\mN$ and $k$ not in the interval $[2-\frac{M}{2},2-M]$, we consider $f_{l,p,q}\cH_p^b\otimes\cH_q^f\subset\cH_k$, so $k=2l+p+q$. If $M\not\in-2\mN$, $p+q+l+\frac{M}{2}-1$ is never zero. If $M\in-2\mN$ and the relation $p+q+l=1-\frac{M}{2}$ would hold, this would imply that, since $l>0$ and $p+q\ge0$,
\begin{eqnarray*}
k=2l+p+q\ge 2-\frac{M}{2}&\mbox{and}&k=2l+p+q\le 2-M
\end{eqnarray*}
hold. Therefore the arrows in Lemma \ref{arrows} will always exist for $M\not\in-2\mN$ and for $M\in-2\mN$ with $k\not\in[2-\frac{M}{2},2-M]$. Therefore we can prove that for every such $u$ there is an element $Y\in\cU(\mathfrak{g})$ such that $Yu=H_k^b$ holds, by using arrows \textbf{(1)} and \textbf{(2)} similarly as \textbf{(3)} and \textbf{(4)} before. This yields the irreducibility for those cases.

Now consider the case $m=1$, when $k\le n$. Taking into account that $H_p^b=0$ when $p>1$, the decomposition diagram above Lemma \ref{arrows} looks like
\[
\xymatrix@=7pt@!C{
&&(0,1,k-1)\ar[d]_{\textbf{(4)}}&(0,k,0)\ar[l]_{\textbf{(2)}}\\
&\cdots&(1,0,k-2)\\
&(\lfloor\frac{k}{2}\rfloor-1,1,\nu(k)+1)\ar[d]_{\textbf{(4)}}&(\lfloor\frac{k}{2}\rfloor-1,0,\nu(k)+2)\ar[l]_{\textbf{(2)}}&\\
\left[(\frac{k-1}{2},1,0)\right]&(\lfloor\frac{k}{2}\rfloor,0,\nu(k))\ar[l]_{\textbf{(2)}}\\
}
\]

where $\nu(k)$ is equal to $0$ if $k$ is even and equal to $1$ if $k$ is odd. The last part is between square brackets since it only exists if $k$ is odd. From this diagram it is immediately clear that $\cH_k$ will be irreducible. The case $n<k\le 2n+1$ can be treated similarly.
\end{proof}

\begin{remark}
The proof above shows that $\cH_k$ (for all values of $k$ and $m|2n$) always corresponds to  the quotient of a Verma module. 
\end{remark}

Since $\Delta_{LB}$ is the quadratic Casimir operator of $\osp$ (see equation \eqref{LBosp}), vectors with a different eigenvalue can never be in the same indecomposable representation. The form of the Laplace-Beltrami operator \eqref{LB} shows that $R^{2p}\cH_{k-2p}\subset\cH_k$ is therefore only possible if
\begin{eqnarray*}
(k-2p)(M-2+k-2p)&=&k(M-2+k)
\end{eqnarray*}
or $p=k-1+M/2$. Combining this with equation \eqref{submodule} shows that every possibility $R^{2p}\cH_{k-2p}\subset\cH_k$ allowed by the Laplace-Beltrami operator appears.

We introduce the notation $L_\lambda^{m|2n}$ for the unique irreducible $\osp$-module with highest weight $\lambda$, see e.g. \cite{MR051963}, even though we will not always use the distinguished root system of \cite{MR051963}.  The irreducible $\mathfrak{sl}_2$-module with lowest weight $j$ is denoted by $L^{(j)}$.

\subsection{The case $m=1$}
For the simple root system of $\mathfrak{osp}(1|2n)$ we choose the standard simple root system of $\mathfrak{sp}(2n)$. This corresponds to the distinguished root system in \cite{MR051963}. The only non-zero spaces of spherical harmonics are $\cH_k$ for $k\le 2n+1$, which follows from Theorem \ref{decompintoirreps} since $\cH_p^b=0$ for $p>1$. If $k\le n$, the highest weight vector of $\cH_k$ for $\mathfrak{osp}(1|2n)$ is the highest weight vector of $\cH_k^f$ for $\mathfrak{sp}(2n)$. For $k> n$ (and $k\le 2n+1$), the highest weight vector is the highest weight vector of 
\begin{eqnarray*}
f_{k-n-1,1,2n-k+1}\,x_1\,\cH_{2n-k+1}^f
\end{eqnarray*}
for $\mathfrak{sp}(2n)$. This can be seen from the fact that $(k-n-1,2n-k+1)$ is the unique pair $(j,l)$ subject to $j+l\le n$ and $2j+l=k$ or $2j+l+1=k$ for which $l$ is maximal. This implies that the highest weight of $\cH_k$ is $(1,\cdots,1,0,\cdots,0)$ where the integer $1$  is repeated $k$ times for $k\le n$ and $2n-k+1$ times for $k>n$. Hence, we obtain
\begin{eqnarray}
\label{hkLk1}
\cH_k&\cong&L^{1|2n}_{(1,\cdots,1,0,\cdots,0)}
\end{eqnarray}
with $L^{1|2n}_{(1,\cdots,1,0,\cdots,0)}$ the irreducible $\mathfrak{osp}(1|2n)$ representation with highest weight $(1,\cdots,1,0,\cdots,0)$, where the integer $1$ has to be repeated as described above. The highest weight of $\cH_k$ in the standard notation corresponds to $\delta_1+\delta_2+\cdots+\delta_k$ if $k\le n$ and $\delta_1+\delta_2+\cdots+\delta_{2n-k+1}$ if $k>n$.

 For the polynomials on the supersphere this implies
\begin{eqnarray*}
\cH\cong\bigoplus_{k=0}^{2n+1}\cH_k\cong\bigoplus_{k=0}^n L^{1|2n}_{(\underline{1}_k,\underline{0}_{n-k})}\oplus\bigoplus_{k=n+1}^{2n+1}L^{1|2n}_{(\underline{1}_{2n-k+1},\underline{0}_{k-1-n})}.
\end{eqnarray*}

The previous results also determine the representation $S(V)_k=\odot^k V$ of supersymmetric tensors of degree $k$ for the natural module $V$ for $\mathfrak{osp}(1|2n)$. This representation can be identified with $\cP_k$. For convenience we consider the cases $k$ even and odd separately. The Fischer decomposition \eqref{superFischer} and the fact that $\cH_k=0$ for $k>2n+1$ imply
\begin{eqnarray*}
S(V)_{2p}\cong \cP_{2p}&=&\bigoplus_{j=0}^{\min(n,p)}R^{2p-2j}\cH_{2j}\\
&\cong&\bigoplus_{j=0}^{\min(\lfloor \frac{n}{2}\rfloor,p)}L^{1|2n}_{(\underline{1}_{2j},\underline{0}_{n-2j})}\oplus\bigoplus_{j=\lfloor \frac{n}{2}\rfloor+1}^{\min( n,p)}L^{1|2n}_{(\underline{1}_{2n-2j+1},\underline{0}_{2j-1-n})}
\end{eqnarray*}
and
\begin{eqnarray*}
S(V)_{2p+1}\cong \cP_{2p+1}&=&\bigoplus_{j=0}^{\min(n,p)}R^{2p-2j}\cH_{2j+1}\\
&\cong&\bigoplus_{j=0}^{\min(\lfloor \frac{n-1}{2}\rfloor,p)}L^{1|2n}_{(\underline{1}_{2j+1},\underline{0}_{n-2j-1})}\oplus\bigoplus_{j=\lfloor \frac{n+1}{2}\rfloor}^{\min( n,p)}L^{1|2n}_{(\underline{1}_{2n-2j},\underline{0}_{2j-n})}.
\end{eqnarray*}

Finally, decomposition \eqref{superFischer} also leads to the following conclusion. Under the joint action of $\mathfrak{sl}_2\times\mathfrak{osp}(1|2n)$, the space $\cP=S(V)$ is isomorphic to the multiplicity free irreducible direct sum decomposition
\begin{eqnarray*}
\cP&\cong &\bigoplus_{k=0}^{n} L^{(k+1/2-n)}\otimes L^{1|2n}_{(\underline{1}_k,\underline{0}_{n-k})}\oplus \bigoplus_{k=n+1}^{2n+1} L^{(k+1/2-n)}\otimes L^{1|2n}_{(\underline{1}_{2n-k+1},\underline{0}_{k-1-n})}.
\end{eqnarray*}

\subsection{The case $m-2n\not\in-2\mN$ with $m>1$}

The simple root system of $\mathfrak{osp}(m|2n)$ as chosen in \cite{MR2395482} is used. This differs from the standard simple root system (see e.g. \cite{MR1773773}) except for $\mathfrak{osp}(2|2n)$, but is more appropriate for the type of representations we will study. The positive odd roots are $\epsilon_j+\delta_i$, $\delta_i$ (in case $m$ is odd) and $\epsilon_j-\delta_i$ (instead of $\delta_i-\epsilon_j$ as in \cite{MR1773773, MR051963}). The connection with the standard root system will be made in Remark \ref{stanroot}.

The highest weight vector of $\cH_k$ for $\mathfrak{osp}(m|2n)$ is therefore the highest weight vector of $\cH_k^b$ for $\mathfrak{so}(m)$, which has weight $(k,0,\cdots,0)$, where $0$ is repeated $\lfloor m/2\rfloor -1$ times. This leads to the highest weight $(k,0,\cdots,0)$ for $\cH_k$ as an $\mathfrak{osp}(m|2n)$-representation, where $0$ is repeated $\lfloor m/2\rfloor +n-1$ times. Theorem \ref{irrH} therefore implies
\begin{eqnarray}
\label{hkLk2}
\cH_k&\cong&L^{m|2n}_{(k,0,\cdots,0)},
\end{eqnarray}
or the highest weight of $\cH_k$ is $k\epsilon_1$. A similar reasoning as in the previous section yields
\begin{eqnarray*}
\cH\cong\bigoplus_{k=0}^{\infty}\cH_k\cong\bigoplus_{k=0}^\infty L^{m|2n}_{(k,0,\cdots,0)}
\end{eqnarray*}
for the polynomials on the supersphere and 
\begin{eqnarray*}
S(V)_{k}\cong \cP_k&=&\bigoplus_{j=0}^{\lfloor\frac{k}{2}\rfloor}R^{2j}\cH_{k-2j}\cong\bigoplus_{j=0}^{\lfloor k/2\rfloor}L^{m|2n}_{k-2j,0,\cdots,0},
\end{eqnarray*}
for $V$ the $m|2n$-dimensional super vectorspace and $S(V)$ the supersymmetric tensor powers. Finally, under the joint action of $\mathfrak{sl}_2\times\mathfrak{osp}(m|2n)$, the space $\cP=S(V)$ is isomorphic to the multiplicity free irreducible direct sum decomposition
\begin{eqnarray*}
\cP&\cong &\bigoplus_{k=0}^\infty L^{(k+M/2)}\otimes L^{m|2n}_{(k,0,\cdots,0)}.
\end{eqnarray*}
The decomposition also satisfies the property that for each distinct representation of $\mathfrak{sl}_2$ there corresponds exactly one irreducible representation of $\osp$, which is required for Howe dual pairs.

\subsection{The case $m-2n\in-2\mN$}
In case $m-2n\in-2\mN$ there is no Fischer decomposition as in Lemma \ref{superFischerLemma}. This also implies that the spaces of spherical harmonics $\cH=\oplus_{k=0}^\infty\cH_k$ do not necessarily correspond to the polynomials on the supersphere. The spaces which do correspond to the polynomials on the supersphere are given by $\cP_k/(R^2\cP_{k-2})$. These spaces also form $\osp$-modules, since $\osp$ commutes with $R^2$.

\begin{theorem}
\label{repPP}
Consider the polynomials on $\cP=\bigoplus_{k=0}^\infty\cP_k$ on $\mR^{m|2n}$ and the spaces $\cP_k/(R^2\cP_{k-2})$ of polynomials on the supersphere $\mS^{m-1|2n}$. The $\osp$-representation $\cP_k/(R^2\cP_{k-2})$ is always indecomposable. It is isomorphic to $\cH_k$ when $\cH_k$ is irreducible. When $\cH_k$ is reducible, so is $\cP_k/(R^2\cP_{k-2})$, but then $\cP_k/(R^2\cP_{k-2})\not\cong\cH_k$.
\end{theorem}
\begin{proof}
The theorem is trivial for $M\not\in-2\mN$, so we focus on $M\in-2\mN$. First we consider the cases $k>2-M $ or $k< 2-\frac{M}{2}$. Since then $\cH_k\cap R^2\cP_{k-2}=0$ (Lemma 5.6 in \cite{Mehler} or a direct consequence of the irreducibility of $\cH_k$) and $\dim\cH_k=\dim\left(\cP_{k}/(R^2\cP_{k-2})\right)$, see formula \eqref{dimHk}, it follows immediately that $\cP_k/(R^2\cP_{k-2})\cong \cH_k$. 

In case $2-\frac{M}{2}\le k\le 2-M$, the subrepresentation of $\cP_k/(R^2 \cP_{k-2})$ generated by $\cH_k\subset \cP_k$ is a proper subrepresentation since $\cH_k\cap R^2\cP_{k-2}$ is not zero, see equation \eqref{submodule}. Therefore, when $\cH_k$ is reducible, so is $\cP_k/(R^2\cP_{k-2})$. It also follows that $\cP_k/(R^2\cP_{k-2})$ is not the quotient of a Verma module, because there is a submodule which contains the highest weight vector. This implies $\cP_k/(R^2\cP_{k-2})\not\cong \cH_k$.

We may decompose $\cP_k/(R^2\cP_{k-2})$ into irreducible $\mathfrak{so}(m)\oplus\mathfrak{sp}(2n)$-representations by a similar procedure as applied to $\cH_k$ in Theorem \ref{decompintoirreps}. This yields
\begin{eqnarray*}
\cP_k/(R^2\cP_{k-2})&=&\bigoplus_{j,l}\left(\theta^{2j}\,\cH_{k-2j-l}^b\otimes\cH^f_l+R^2\cP_{k-2}\right),
\end{eqnarray*}
which is a consequence of the purely bosonic and fermionic Fischer decompositions in Lemma \ref{superFischerLemma}. Again we can construct four arrows from which it will follow that $\cP_k/(R^2\cP_{k-2})$ is indecomposable. The procedure is identical to the proof of Theorem \ref{irrH}, but the calculations are simpler.
\end{proof}

\subsection{The Cartan product}

One of the results in \cite{MR2130630} is that for $V$ a finite dimensional irreducible $\mathfrak{sl}(n)$-module, the property
\begin{eqnarray*}
\circledcirc^kV&=&\left(\circledcirc^{k-1}V\right)\otimes V\cap V\otimes\left(\circledcirc^{k-1}V\right)
\end{eqnarray*}
holds, for $\circledcirc^kV$ the irreducible representation inside $\otimes^kV$ with highest weight known as the Cartan product. In \cite{Joseph} it will be proved that this property can be extended to arbitrary semisimple Lie algebras, $\mathfrak{gl}(p|q)$ and $\mathfrak{osp}(2|2n)$. In the case of other Lie superalgebras tensor products of simple modules are not always semisimple and the Cartan product not always defined. The results in this Section show that another relation holds for the natural representation of $\osp$. The proposed property is
\begin{eqnarray*}
\cU(\mg)\cdot v_+^{\otimes k}&=&\left(\cU(\mg)\cdot  v_+^{\otimes k-1}\right)\otimes\mg\cap \mg\otimes\left(\cU(\mg)\cdot  v_+^{\otimes k-1}\right)
\end{eqnarray*}
for $v_+$ the highest weight vector of the representation $V$ and $\mg$ the Lie superalgebra. It follows from the identification of traceless symmetric tensor with spherical harmonics and the results in this section that this equation holds in case $V=\mC^{m|2n}$ and $\mg=\osp$. In \cite{Joseph} this conjecture will be discussed for other cases.

\subsection{The algebras $\mathfrak{osp}(4n|2m)$ and $\mathfrak{osp}(4n+1|2m)$}

In the theory of Howe dual pairs the two dual algebras are each other's centralizers inside a bigger algebra, such that the relevant representation (which has a multiplicity free decomposition into irreducible pieces under the action of the dual pair) constitutes an irreducible representation of the big algebra.

For the Fischer decomposition of polynomials on $\mR^m$ this is $\mathfrak{so}(m)\times \mathfrak{sl}_2\subset \mathfrak{sp}(2m)$. The polynomials $\mR[x_1,\cdots, x_m]$ decompose into two irreducible $\mathfrak{sp}(2m)$-representations, corresponding to the even and odd polynomials. Thus there are two non-isomorphic representations corresponding to the Howe dual pair $\mathfrak{sp}(2m)\supset \mathfrak{so}(m)\times \mathfrak{sl}_2$. To obtain the polynomials as one irreducible representation, $\mathfrak{sp}(2m)$ needs to be embedded inside the superalgebra $\mathfrak{osp}(1|2m)$. In this subsection we obtain the corresponding results for $\mathfrak{sl}_2\times\mathfrak{osp}(m|2n)$.

\begin{theorem}
The Lie superalgebra spanned by the differential operators 
\[X_i(-1)^{(1-[i])\mE},\quad \partial_{X_i}(-1)^{(1-[i])\mE},\]
\[ X_iX_j(-1)^{([i]+[j])\mE},\quad\partial_{X_i}\partial_{X_j}(-1)^{([i]+[j])\mE}\quad \mbox{and}\quad X_i\partial_{X_j}(-1)^{([i]+[j])\mE}+\frac{(-1)^{[i]}\delta_{ij}}{2}
\]
where the generators $X_i(-1)^{(1-[i])\mE}$ and $\partial_{X_i}(-1)^{(1-[i])\mE}$ have gradation $1-[i]$, is isomorphic to $\mathfrak{osp}(4n+1|2m)$. If the operators $X_i(-1)^{(1-[i])\mE}$ and $\partial_{X_i}(-1)^{(1-[i])\mE}$ are removed and only the quadratic elements are used the Lie superalgebra $\mathfrak{osp}(4n|2m)$ is obtained. With respect to these realizations, the space of polynomials on $\mR^{m|2n}$ satisfies
\begin{eqnarray*}
\cP&\cong&L^{4n+1|2m}_{(\frac{1}{2},\cdots, \frac{1}{2},-\frac{1}{2},\cdots,-\frac{1}{2})}\qquad \mbox{where $1/2$ is repeated $2n$ times and $-1/2$ is repeated $m$ times}\\
\cP&\cong&L^{4n|2m}_{(\frac{1}{2},\cdots, \frac{1}{2},-\frac{1}{2},\cdots,-\frac{1}{2})}\,\oplus\,L^{4n|2m}_{(\frac{1}{2},\cdots, \frac{1}{2},-\frac{1}{2},\cdots,-\frac{1}{2},-\frac{3}{2})},
\end{eqnarray*}
where the first irreducible representation corresponds to the polynomials of even homogeneous degrees and the second to those of odd degrees.
\end{theorem}
\begin{proof}
Adding the operator $(-1)^{(1-[i])\mE}$ to the variables changes them so that the $P_j=X_i(-1)^{(1-[i])\mE}$ satisfy the relation
\begin{eqnarray*}
P_{i}P_j&=&(-1)^{[i]+[j]+[i][j]}P_jP_i=-(-1)^{(1-[i])(1-[j])}P_jP_i,
\end{eqnarray*}
which implies that the commuting and anti-commuting variables now mutually anti-commute. In fact, with the gradation where $P_i$ corresponds to $|P_i|=1-[i]$, we now have a super anti-commutative algebra, since $P_iP_j=-(-1)^{|P_i||P_j|}P_jP_i$. This then equals the oscillator realization of the orthosymplectic Lie superalgebra, see Chapter 29 in \cite{MR1773773} and equals to the super spinor or super metaplectic realizations in \cite{Tensor}. 
\end{proof}

The realization of $\mathfrak{sl}_2$ as differential operators on $\cP$ is clearly embedded in this realization of $\mathfrak{osp}(4n+1|2m)$ and $\mathfrak{osp}(4n|2m)$. Instead of the $\osp$-action of before in equation \eqref{ospgen} we obtain
\begin{eqnarray*}
\left(X_i\partial_{X^j}-(-1)^{[i][j]}X_j\partial_{X^i}\right)(-1)^{([i]+[j])\mE}&=&L_{ij}(-1)^{([i]+[j])\mE}\in\mathfrak{osp}(4n|2m)\subset\mathfrak{osp}(4n+1|2m).
\end{eqnarray*}
Since $\mE$ and $L_{ij}$ commute these operators satisfy the same commutation relations as $L_{ij}$ and thus generate $\osp$. It is clear that this representation of $\osp$ on $\cP$ has the exact same properties as the one we studied before.

So we have obtained the bigger algebra in which $\osp$ and $\mathfrak{sl}_2$ are embedded
\begin{eqnarray*}
\osp\times \mathfrak{sl}_2&\subset& \mathfrak{osp}(4n+1|2m),
\end{eqnarray*}
for which $\cP$ is an irreducible representation.

When considering the Lie superalgebra $\mathfrak{osp}(4n|2m)$ we obtain
\begin{eqnarray*}
\mathfrak{osp}(4n|2m)\supset\mathfrak{osp}(m|2n)\times \mathfrak{sl}_2&\mbox{ as a generalization of}& \mathfrak{sp}(2m)\supset\mathfrak{so}(m)\times \mathfrak{sl}_2,
\end{eqnarray*}
where, $\mathfrak{osp}(m|2n)$ and $\mathfrak{sl}_2$ are each other's centralizers inside $\mathfrak{osp}(4n|2m)$, see \cite{MR1893457}. All obtained information on this Howe duality is summarized in the following theorem.
\begin{theorem}
\label{AllHowe}
Two non-isomorphic realizations for the Howe duality $\mathfrak{osp}(4n|2m)\supset\mathfrak{sl}_2\times \osp$ are given by
\begin{eqnarray*}
L^{4n|2m}_{(\frac{1}{2},\cdots, \frac{1}{2},-\frac{1}{2},\cdots,-\frac{1}{2})}\,\cong\, \bigoplus_{j=0}^\infty \odot^{2j}L^{m|2n}_{(1,0,\cdots,0)}&\cong&\bigoplus_{k=0}^\infty L^{(2k+M/2)}\otimes L^{m|2n}_{(2k,0,\cdots,0)}\\
L^{4n|2m}_{(\frac{1}{2},\cdots, \frac{1}{2},-\frac{1}{2},\cdots,-\frac{1}{2},-\frac{3}{2})}\,\cong\, \bigoplus_{j=0}^\infty \odot^{2j+1}L^{m|2n}_{(1,0,\cdots,0)}&\cong&\bigoplus_{k=0}^\infty L^{(2k+1+M/2)}\otimes L^{m|2n}_{(2k+1,0,\cdots,0)},
\end{eqnarray*} 
if $m-2n\not\in-2\mN$.
\end{theorem}

The Howe duality breaks down in case $m-2n\in-2\mN$, but as mentioned before still holds in the case $m=0$. Then we obtain the classical Howe duality $\mathfrak{sp}(2n)\times\mathfrak{sl}_2\subset \mathfrak{so}(4n)$, so the Howe dual pair $\mathfrak{osp}(m|2n)\times \mathfrak{sl}_2\subset \mathfrak{osp}(4n|2m)$ is also a generalization of this classical case.

\section{The representations $L_{(k,0,\cdots,0)}^{m|2n}$}
In case $m=1$, the representation $\cH_k\cong L^{1|2n}_{(1,\cdots,1,0,\cdots,0)}$ is typical since all $\mathfrak{osp}(1|2n)$-representations are, see \cite{MR051963}. Their dimension and decomposition with respect to the underlying Lie algebra is therefore well-known, see e.g. \cite{MR1773773}. In case $m>1$ the representation $L_{(k,0,\cdots,0)}^{m|2n}$ is always atypical except for $m=2$ and $k=n$ or $k>2n$. This can be concluded from the atypicality conditions in Chapter $36$ of \cite{MR1773773} by using the highest weight in the distinguished root system.

\begin{remark}
\label{stanroot}
In the distinguished positive root system (see e.g. \cite{MR1773773}), the highest weight of the representation $\cH_k$ would be given by 
\[\delta_1+\cdots+\delta_{\min(k,n)}+(k-\min(k,n))\epsilon_1\]
instead of $k\epsilon_1$. This shows it is more elegant to use the choice of positive roots made in \cite{MR2395482} for this kind of representations.
\end{remark}

In this section we will assume $m>1$ and obtain information on the representations $L_{(k,0,\cdots,0)}^{m|2n}$ based on the results on spherical harmonics in Section \ref{SectionPolySS}. First of all we need to realize the representations $L_{(k,0,\cdots,0)}^{m|2n}$ by taking the appropriate quotient of $\cH_k$ in case $M\in-2\mN$ and $2-\frac{M}{2}\le k\le 2-M$. Since the representations $L_{(k,0,\cdots,0)}^{m|2n}$ are atypical, their dimensions do not follow from the standard formula and their decompositions as $\sosp$-representations are not known, they are obtained from the following theorem.

\begin{theorem}
\label{Lkrep}
If $M=m-2n\in-2\mN$ and $2-\frac{M}{2}\le k\le 2-M$, the space $R^{2k+M-2}\cH_{2-M-k}$ is irreducible and is the maximal $\osp$-submodule of $\cH_k$. This implies that in that case,
\begin{eqnarray}
\label{voordimLk}
R^2\cP_{k-2}\cap\cH_k&=&R^{2k+M-2}\cH_{2-M-k}
\end{eqnarray}
holds. Then, for every $(m,n,k)\in\mN^3$ with $m>1$, the relation 
\begin{eqnarray*}
L_{(k,0,\cdots,0)}^{m|2n}&\cong& \cH_k/\left(\cH_k\cap R^2\cP_{k-2}\right)
\end{eqnarray*}
holds for $L_{(k,0,\cdots,0)}^{m|2n}$ the simple $\osp$-module with highest weight $(k,0,\cdots,0)$ and $\cH_k$ the spherical harmonics on $\mR^{m|2n}$ of homogeneous degree $k$. 

The decomposition into irreducible representation of $L_{(k,0,\cdots,0)}^{m|2n}$ under the action of $\sosp$ if $M\not\in-2\mN$ or $M=-2p$ and $k\not\in[2+p,2+2p]$ is given by
\begin{eqnarray}
\label{decompLk}
L^{m|2n}_{(k,0,\cdots,0)}&\cong &\bigoplus_{j=0}^{\min(n, k)}\bigoplus_{l=0}^{\min(n-j,\lfloor \frac{k-j}{2} \rfloor)} L^{m|0}_{(k-2l-j,0,\cdots,0)} \otimes L^{0|2n}_{(\underline{1}_j,\underline{0}_{n-j})}\quad\mbox{and}\\
\label{decompLkatyp}
L^{m|2n}_{(k,0,\cdots,0)}&\cong &\bigoplus_{j=0}^{\min(n, k)} \bigoplus_{l=0}^{\min(n-j,\lfloor \frac{k-j}{2}\rfloor,k-p-2 )}L^{m|0}_{(k-2l-j,0,\cdots,0)} \otimes L^{0|2n}_{(\underline{1}_j,\underline{0}_{n-j})}
\end{eqnarray}
if $M=-2p$ and $k\in[2+p,2+2p]$.
\end{theorem}
\begin{proof}
First we restrict to $M\in-2\mN$ and $k\in[2-\frac{M}{2},2-M]$. As mentioned in Equation \eqref{submodule}, then $R^{2k+M-2}\cH_{2-M-k}$ is a submodule. Since the inequality $2-M-k\le -\frac{M}{2}<2-\frac{M}{2}$ holds, Theorem \ref{irrH} implies that $R^{2k+M-2}\cH_{2-M-k}$ is an irreducible $\mathfrak{osp}(m|2n)$-representation. 

The $\osp$-submodule $R^{2k+M-2}\cH_{2-M-k}$ corresponds to all the $\mathfrak{so}(m)\oplus\mathfrak{sp}(2n)$-submodules of $\cH_k$ of the form 
\begin{eqnarray*}
f_{l,k-2l-j,j}\cH_{k-2l-j}^b\otimes\cH_j^f
\end{eqnarray*}
with $l\ge k+\frac{M}{2}-1$. This follows from Equation \eqref{submodule} and the unicity of $f_{l,p,q}$ in Lemma \ref{polythm} which implies 
\begin{eqnarray*}
R^{2k+M-2}f_{i,2-k-M-2i-j,j}\sim f_{k+\frac{M}{2}-1+i,2-k-M-2i-j,j}.
\end{eqnarray*}
Therefore, as an $\mathfrak{so}(m)\oplus\mathfrak{sp}(2n)$-representations the $\osp$-subrepresentation $R^{2k+M-2}\cH_{2-M-k}$ in $\cH_k$ has a complement $\cH_k'$,
\begin{eqnarray*}
\cH_k&=&R^{2k+M-2}\cH_{2-M-k}\oplus \cH_k',
\end{eqnarray*}
which is given by
\begin{eqnarray*}
\cH_k'&=&\bigoplus_{j=0}^{\min(n, k)} \bigoplus_{l=0}^{\min(n-j,\lfloor \frac{k-j}{2} \rfloor, k+\frac{M}{2}-2)} f_{l,k-2l-j,j} \cH^b_{k-2l-j} \otimes \cH^f_{j}.
\end{eqnarray*}

If there would exist a submodule larger than $R^{2k+M-2}\cH_{2-M-k}$, it would include one of the $\mathfrak{so}(m)\oplus \mathfrak{sp}(2n)$-modules in the decomposition of $\cH_k'$ above. For the spaces in that decomposition arrow \textbf{(1)} in Lemma \ref{arrows} always exists. Such a submodule would therefore include $\cH_k^b$ (by using arrows \textbf{(1)} and \textbf{(2)} consecutively). From the proof of Theorem \ref{irrH} it is clear that a submodule of $\cH_k$ containing $\cH_k^b$ is always equal to $\cH_k$ for any value of $(m,n,k)$. 

The relation $L_{(k,0,\cdots,0)}^{m|2n}\cong \cH_k/\left(\cH_k\cap R^2\cP_{k-2}\right)$ for arbitrary $k$ and $M$ then follows immediately.

The combination of Theorem \ref{decompintoirreps} with the considerations above also yields the decomposition of $L^{m|2n}_{(k,0,\cdots,0)}$ as an $\sosp$-representation.
\end{proof}

The dimensions of $L_{(k,0,\cdots,0)}^{m|2n}$ then immediately follow from Equation \eqref{voordimLk} and the dimensions of the spaces $\cH_k$ in Equation \eqref{dimHk}. 
In case $M=m-2n\in-2\mN$ and $2-\frac{M}{2}\le k\le 2-M$, the dimension of $L_{(k,0,\cdots,0)}^{m|2n}$ is given by
\begin{eqnarray*}
\dim L_{(k,0,\cdots,0)}^{m|2n}&=&\sum_{i=0}^{\min(k,2n)}\binom{2n}{i}\binom{k-i+m-1}{m-1}-\sum_{i=0}^{\min(k-2,2n)}\binom{2n}{i}\binom{k-i+m-3}{m-1}\\
&+&\sum_{i=0}^{\min(-M-k,2n)}\binom{2n}{i}\binom{2n-k-i-1}{m-1}-\sum_{i=0}^{\min(2-M-k,2n)}\binom{2n}{i}\binom{2n-k-i+1}{m-1}.
\end{eqnarray*}
In the other cases $\cH_k\cong L_{(k,0,\cdots,0)}^{m|2n}$ and
\begin{eqnarray*}
\dim L_{(k,0,\cdots,0)}^{m|2n}&=&\sum_{i=0}^{\min(k,2n)}\binom{2n}{i}\binom{k-i+m-1}{m-1}-\sum_{i=0}^{\min(k-2,2n)}\binom{2n}{i}\binom{k-i+m-3}{m-1}
\end{eqnarray*}
holds.

In \cite{MR0621253} the representations $L_{(k,0,\cdots,0)}^{m|2n}$ were constructed as the Cartan product inside tensor products of the form $\otimes^k(L_{(1,0,\cdots,0)}^{m|2n})$. Here, the Cartan product corresponds to the traceless supersymmetric part, which can be identified with $\cP_k/(R^2\cP_{k-2})$. However as has been shown in Theorem \ref{repPP}, this construction only holds when $m-2n\not\in-2\mN$, due to the lack of complete reducibility. This was overlooked in the formal approach in \cite{MR0621253} were for instance the number $m-2n$ appears as a pole in Equation $(4.21)$. The correct construction of $L^{m|2n}_{(k,0,\cdots,0)}$ inside the supersymmetric tensor products of $V=L_{(1,0,\cdots,0)}^{m|2n}$ is given in Theorem \ref{Lkrep}.

In \cite{MR2395482} it was proved that for $M=m-2n>2$ the following branching rule holds for $\mathfrak{osp}(m-1|2n)\hookrightarrow \osp$:
\begin{eqnarray}
\label{branching}
L_{(k,0,\cdots,0)}^{m|2n}&\cong&\bigoplus_{l=0}^kL_{(l,0,\cdots,0)}^{m-1|2n} \qquad\mbox{as an $\mathfrak{osp}(m-1|2n)$-module.}
\end{eqnarray}
This can immediately be extended to $m-2n=2$, but not to the case $m-2n\le 1$ because of the appearance of not completely reducible representations. Using the insights developed in this paper it is now possible to calculate the branching rules for $m-2n\le 1$ as well. This is given in the following theorem where we obtain the branching rules for all the cases where $L_{(k,0,\cdots,0)}^{m|2n}$ is a completely reducible $\mathfrak{osp}(m-1|2n)$-representation.

\begin{theorem}
\label{branchingThm}
In case $m-2n>1$, or $m-2n\in 1-2\mN$ with $k<2+\frac{1-m}{2}+n$, or $m-2n\in-2\mN$ with $k<2-\frac{m}{2}+n$ or $k>2-m+2n$, the branching rule
\begin{eqnarray*}
L_{(k,0,\cdots,0)}^{m|2n}&\cong&\bigoplus_{l=0}^kL_{(l,0,\cdots,0)}^{m-1|2n} \qquad\mbox{as an $\mathfrak{osp}(m-1|2n)$-module},
\end{eqnarray*}
holds. In case $m-2n\in-2\mN$ and $2-\frac{m}{2}+n\le k\le 2-m+2n$, the branching rule
\begin{eqnarray*}
L_{(k,0,\cdots,0)}^{m|2n}&\cong&\bigoplus_{l=3-m+2n-k}^kL_{(l,0,\cdots,0)}^{m-1|2n} \qquad\mbox{as an $\mathfrak{osp}(m-1|2n)$-module},
\end{eqnarray*}
holds. In the other cases ($m-2n\in1-2\mN$ with $k\ge 2+\frac{1-m}{2}+n$), $L_{(k,0,\cdots,0)}^{m|2n}$ is not completely reducible as an $\mathfrak{osp}(m-1|2n)$-representation.
\end{theorem}

\begin{proof}
Since $\cP_k=\oplus_{l=0}^kx_1^{k-l}\cP'_{l}$ where $\cP'$ denotes the space of polynomials on $\mR^{m-1|2n}$, we obtain
\begin{eqnarray*}
\cP_k/(R^2\cP_{k-2})&\cong&\bigoplus_{l=0}^k\left(\cP'_l/(R_1^2\cP'_{l-2})\right)\qquad\mbox{as an $\mathfrak{osp}(m-1|2n)$-module, for general $M$,}
\end{eqnarray*}
with $R_1^2$ the generalized norm squared on $\mR^{m-1|2n}$, $R^2=x_1^2+R_1^2$. This leads to the following two conclusions based on Theorem \ref{irrH}, Theorem \ref{Lkrep} and Theorem \ref{repPP}:
\begin{itemize}
\item If $m-2n\le 1$ but $m-2n\not\in-2\mN$, $L_{(k,0,\cdots,0)}^{m|2n}$ is not completely reducible as an $\mathfrak{osp}(m-1|2n)$-representation if $k\ge 2+\frac{1-M}{2}$. Relation \eqref{branching} still holds if $k< 2+\frac{1-M}{2}$.
\item If $m-2n\in-2\mN$, the space $\cP_k/(R^2\cP_{k-2})$ which is not necessarily completely reducible as an $\osp$-representation decomposes into irreducible $\mathfrak{osp}(m-1|2n)$-representations as 
\begin{eqnarray}
\label{PLk}
\cP_k/(R^2\cP_{k-2})&\cong& \bigoplus_{l=0}^kL_{(l,0,\cdots,0)}^{m-1|2n}.
\end{eqnarray}
Therefore, Equation \eqref{branching} still holds if $k<2-\frac{M}{2}$ or $k>2-M$.
\end{itemize}

This already proves all the results except the case $m-2n\in-2\mN$ and $2-\frac{m}{2}+n\le k\le 2-m+2n$. Theorem \ref{Lkrep} and Equation \eqref{PLk} imply that for that case
\begin{eqnarray*}
L_{(k,0,\cdots,0)}^{m|2n}\cong \cH_k/\left(\cH_k\cap R^2\cP_{k-2}\right) \subset \cP_k/(R^2\cP_{k-2}) \cong \bigoplus_{l=0}^kL_{(l,0,\cdots,0)}^{m-1|2n}
\end{eqnarray*}
holds. This implies that 
\begin{eqnarray}
\label{knownformbranch}
L_{(k,0,\cdots,0)}^{m|2n}&\cong&\bigoplus_{p\in I}L_{(p,0,\cdots,0)}^{m-1|2n} \qquad\mbox{as an $\mathfrak{osp}(m-1|2n)$-module,}
\end{eqnarray}
with $I\subset\{0,\cdots,k\}$. We look at the decomposition of $L_{(k,0,\cdots,0)}^{m|2n}$ into simple $\mathfrak{so}(m)\oplus \mathfrak{sp}(2n)$-modules. Equation \eqref{decompLkatyp} implies that the $\mathfrak{sp}(2n)$-trivial the part of this decomposition is given by
\begin{eqnarray*}
L_{(k,0,\cdots,0)}^{m|2n}&\rightarrow&\bigoplus_{l=0}^{k+\frac{M}{2}-2}L_{(k-2l,0,\cdots,0)}^{m|0},
\end{eqnarray*}
since $\min(n,\lfloor \frac{k}{2}\rfloor,k+\frac{M}{2}-2)=k+\frac{M}{2}-2$. Branched to $\mathfrak{so}(m-1)$ this gives $\bigoplus_{l=0}^{k+\frac{M}{2}-2}\bigoplus_{j=0}^{k-2l}L_{(k-2l-j,0,\cdots,0)}^{m-1|0}$. Equation \eqref{decompLk} implies the $\mathfrak{sp}(2n)$-trivial part of $L_{(p,0,\cdots,0)}^{m-1|2n}$ for $p\le k$ is
\begin{eqnarray*}
L_{(p,0,\cdots,0)}^{m-1|2n}&\to& \bigoplus_{l=0}^{\lfloor \frac{p}{2}\rfloor}L_{(p-2l,0,\cdots,0)}^{m-1|0}
\end{eqnarray*}
since $\min(n,\lfloor \frac{p}{2}\rfloor)=\lfloor \frac{p}{2}\rfloor$ if $p\le k\le 2-m+2n$. If Equation \eqref{knownformbranch} holds,  the equation
\begin{eqnarray}
\label{eqforbranching}
\bigoplus_{l=0}^{k+\frac{M}{2}-2}\bigoplus_{j=0}^{k-2l}L_{(k-2l-j,0,\cdots,0)}^{m-1|0}&=&\bigoplus_{p\in I}\bigoplus_{l=0}^{\lfloor \frac{p}{2}\rfloor}L_{(p-2l,0,\cdots,0)}^{m-1|0}
\end{eqnarray}
must hold as well. It turns out that this equation is enough to determine $I$. We introduce the shorthand notation $(q)=L_{(q,0,\cdots,0)}^{m-1|0}$. In the left-hand side of Equation \eqref{eqforbranching} the module $(j)$ appears $1+\lfloor\frac{k-j}{2}\rfloor$ times for $j\ge -k-M+3$. To obtain all these in the right-hand side of Equation \eqref{eqforbranching}, $\{3-M-k,\cdots,k\}\subset I$ must hold. Since the module $(j)$ appears $k+M/2-1$ times for $j\le -k-M+2$ we find  $\{3-M-k,\cdots,k\}= I$, because otherwise there would be too many $\mathfrak{so}(m-1)$-modules in the right-hand side.
\end{proof}

In the statements above we assumed $m\not=2$ and $m\not=1$. The corresponding statements for those dimensions are straightforward.

\subsection*{Acknowledgment}
The author would like to thank Ruibin Zhang, Joris Van der Jeugt and Bent \O rsted for helpful suggestions and comments.



\begin{thebibliography}{ASM}


\bibitem{MR2667819}
{A. Alldridge, J. Hilgert},
\newblock{\it Invariant Berezin integration on homogeneous supermanifolds,}
\newblock {J. Lie Theory} {\bf 20} (2010), 65--91.

\bibitem{MR0621253}
A. Baha Balantekin, I. Bars,
\newblock{\it Dimension and character formulas for Lie supergroups.}
\newblock{J. Math. Phys. 22 (1981), 1149--1162. }

\bibitem{MR1632811}
G. Benkart, C.L. Shader, A. Ram,
\newblock{Tensor product representations for orthosymplectic Lie superalgebras,}
\newblock{J. Pure Appl. Algebra} {130} (1998), 1--48. 


\bibitem{MR2207328}
C. Carmeli, G. Cassinelli, A. Toigo, V.S. Varadarajan,
\newblock{\it Unitary representations of super Lie groups and applications to the classification and multiplet structure of super particles},
\newblock{Comm. Math. Phys.} {\bf263} (2006), 1, 217--258.

\bibitem{MR2441598}
S.J. Cheng, W. Wang,
\newblock{\it Remarks on modules of the ortho-symplectic Lie superalgebras.}
\newblock{Bull. Inst. Math. Acad. Sin. (N.S.)} {\bf3} (2008), 353--372. 

\bibitem{MR2028498}
S.J. Cheng, R.B. Zhang, 
\newblock{\it Howe duality and combinatorial character formula for orthosymplectic Lie superalgebras},
\newblock{Adv. Math.} {\bf182} (2004), 1, 124--172.

\bibitem{Tensor}
K. Coulembier,
\newblock{\it On a class of tensor product representations for the orthosymplectic superalgebra}
\newblock{Accepted in J. Pure Appl. Algebra arXiv:1205.0119.}

\bibitem{MR2539324}
K. Coulembier, H. De~Bie, F. Sommen,
\newblock {\it Integration in superspace using distribution theory},
\newblock {J. Phys. A: Math. Theor.} {\bf 42} (2009), 395206.

\bibitem{CDBS3}
K. Coulembier, H. De Bie, F. Sommen,
\newblock {\it Orthosymplectically invariant functions in superspace},
\newblock {J. Math. Phys.} {\bf 51} (2010), 083504.

\bibitem{Mehler}
{K. Coulembier, H. De~Bie, F. Sommen},
\newblock {\it Orthogonality of Hermite polynomials in superspace and Mehler type formulae.}
\newblock {Proc. London Math. Soc. (2011) 103(5): 786--825.}

\bibitem{Joseph}
{K. Coulembier, P. Somberg, V. Soucek},
\newblock{\it On Joseph-type ideals for $\osp$.}
\newblock{In preparation.}

\bibitem{CZ}
K. Coulembier, R.B. Zhang,
\newblock{\it Invariant integration on orthosymplectic and unitary supergroups.}
\newblock{J. Phys. A: Math. Theor. 45 (2012) 095204 (36 pp).}

\bibitem{DBE1}
H. De~Bie, D. Eelbode, F. Sommen, 
\newblock {\it Spherical harmonics and integration in superspace II},
\newblock {J. Phys. A: Math. Theor.} {\bf 42} (2009), 245204.


\bibitem{MR1827871}
{C. Dunkl, Y. Xu,}
\newblock{\it Orthogonal polynomials of several variables,}
\newblock{Cambridge University Press, Cambridge, 2001. }

\bibitem{MR2130630}
M. Eastwood,
\newblock{\it The Cartan product.}
\newblock{Bull. Belg. Math. Soc. Simon Stevin 11 (2004), 641Ð-651. }


\bibitem{MR1773773}
{L. Frappat, A. Sciarrino, P. Sorba},
\newblock {\it Dictionary on {L}ie algebras and superalgebras},
\newblock Academic Press Inc., San Diego, CA, 2000.  arXiv:hep-th/9607161.

  
\bibitem{MR2434470}
O. Goertsches,
\newblock{\it Riemannian supergeometry},
\newblock{Math. Z.} {\bf260} (2008), 3, 557--593.

\bibitem{MR0986027}
{R. Howe},
\newblock {\it Remarks on classical invariant theory},
\newblock {Trans. Amer. Math. Soc.} {\bf313} (1989), 2, 539--570.

\bibitem{MR0546778}
{P.D. Jarvis, H.S. Green},
\newblock{\it Casimir invariants and characteristic identities for generators of the general linear, special linear and orthosymplectic graded Lie algebras},
\newblock{J. Math. Phys.} {\bf20} (1979), 10, 2115--2122.

\bibitem{MR051963}
{V. Kac,}
\newblock{\it Representations of classical Lie superalgebras}, 
\newblock{Lecture Notes in Math.} {\bf676}, Springer, Berlin, 1978.

\bibitem{MR1893457}
D. Leites, I. Shchepochkina, 
\newblock{\it The Howe duality and Lie superalgebras.}
\newblock{Noncommutative structures in mathematics and physics, 93--111,
NATO Sci. Ser. II Math. Phys. Chem., 22, Kluwer Acad. Publ., Dordrecht, 2001.}

\bibitem{MR1272070}
K. Nishiyama,
\newblock{\it Super dual pairs and highest weight modules of orthosymplectic algebras},
\newblock{Adv. Math.} {\bf104} (1994), 1, 66--89.

\bibitem{MR1897209}
H. Saleur, B. Wehefritz-Kaufmann,
\newblock{\it Integrable quantum field theories with ${\rm OSp}(m/2n)$ symmetries},
\newblock{Nuclear Phys. B} {\bf628} (2002), 3, 407--441.


\bibitem{MR2172158}
{M. Scheunert, R.B. Zhang,}
\newblock{\it Integration on Lie supergroups: a Hopf superalgebra approach},
\newblock{J. Algebra} {\bf292} (2005),  2, 324--342.

\bibitem{MR2125586}
{A.F. Schunck, C. Wainwright},
\newblock{\it A geometric approach to scalar field theories on the supersphere},
\newblock{J. Math. Phys}. {\bf46} (2005), 3, 033511.

\bibitem{MR1827078}
{A. Sergeev},
\newblock{\it An analog of the classical invariant theory for Lie superalgebras},
\newblock{Michigan Math. J.} {\bf49} (2001), 1, 113--168.

\bibitem{MR0922822}
J. Van der Jeugt, 
\newblock{\it Orthosymplectic representations of Lie superalgebras},
\newblock{Lett. Math. Phys.} {\bf14} (1987), 4, 285--291. 

\bibitem{MR2395482}
R.B. Zhang, 
\newblock {\it Orthosymplectic Lie superalgebras in superspace analogues of quantum  Kepler problems},
\newblock {Comm. Math. Phys.} {\bf 280} (2008), 545--562.



\end{thebibliography}
\end{document}